\documentclass[12pt]{article}
\usepackage{graphicx}
\usepackage[all]{xy}
\usepackage{color}
\usepackage{latexsym}
\usepackage{amsmath,amstext}
\usepackage{amssymb}
\usepackage{epsfig}
\usepackage{graphics}
\usepackage{euscript}
\usepackage{ifthen}
\usepackage{wrapfig}
\usepackage{amsthm}
\usepackage{multicol}
\usepackage{paralist}
\usepackage{verbatim}
\usepackage[colorlinks=true, urlcolor=blue, linkcolor=blue, citecolor=blue]{hyperref}
\usepackage[margin=1in]{geometry}
\usepackage{placeins}
\usepackage[linesnumbered,lined,commentsnumbered,ruled]{algorithm2e}
\usepackage{caption}
\usepackage{subcaption}
\expandafter\def\csname ver@subfig.sty\endcsname{}
\usepackage{longtable}
\usepackage{lscape}
\usepackage{fancyhdr}
\usepackage{listings}
\usepackage{mathtools}
\usepackage{subfig}
\usepackage{appendix} %appendices within sections
\usepackage{float}

\usepackage[utf8]{inputenc}
\usepackage{csquotes}% Recommended for biblatex
\usepackage{xparse}% for lists in macros
\usepackage[dvipsnames]{xcolor}
\usepackage[normalem]{ulem}%to strike though text

\theoremstyle{plain}
\newtheorem{theorem}{Theorem}[section]
\newtheorem{cor}[theorem]{Corollary}
\newtheorem{lemma}[theorem]{Lemma}
\newtheorem{prop}[theorem]{Proposition}

\newtheorem{conjecture}[theorem]{Conjecture}
\theoremstyle{definition}
\newtheorem{definition}[theorem]{Definition}
\newtheorem{remark}[theorem]{Remark}
\newtheorem{example}[theorem]{Example}
\newtheorem{question}[theorem]{Question}
\newtheorem{notation}[theorem]{Notation}

\numberwithin{equation}{section}

\newcommand{\R}{\mathbb{R}}

\newcommand{\sS}{\mathcal{S}}

\newcommand{\sN}{\mathcal{N}}

\DeclareMathOperator{\conv}{conv}

\DeclareMathOperator{\vol}{vol}

\allowdisplaybreaks

\title{Exact Solutions in Log-Concave\\ Maximum Likelihood Estimation}

\author{Alexandros Grosdos, Alexander Heaton, Kaie Kubjas, Olga Kuznetsova,\\ Georgy Scholten, and Miruna-Stefana Sorea}

\begin{document}

\maketitle
\begin{abstract}
We study probability density functions that are log-concave. Despite the space of all such densities being infinite-dimensional, the maximum likelihood estimate is the exponential of a piecewise linear function determined by finitely many quantities, namely the function values, or heights, at the data points. We explore in what sense exact solutions to this problem are possible.
First, we show that the heights given by the maximum likelihood estimate are often transcendental. For a cell in one dimension, 
the maximum likelihood estimator is expressed in closed form using the generalized $W$-Lambert function. Even  more, we show that finding the log-concave maximum likelihood estimate is equivalent to solving a collection of polynomial-exponential systems of a special form. 
Even in the case of two equations, very little is known about solutions to these systems. 
As an alternative, we use Smale's $\alpha$-theory to refine approximate numerical solutions and to certify solutions to log-concave density estimation.
\end{abstract}

\tableofcontents

\section{Introduction}
Nonparametric methods in statistics emerged in the 1950-1960s~\cite{fix1951discriminatory,rosenblatt1956estimation, parzen1962estimation, ayer1955empirical} and fall into two main streams: smoothing methods and shape constraints. 
Examples of smoothing methods include delta sequence methods such as kernel, histogram and orthogonal series estimators~\cite{walter1979probability}, and penalized maximum likelihood estimators, e.g., spline methods~\cite{eggermont2001maximum}. Their defining feature is the need to choose the  smoothing or tuning parameters. It is a delicate process because smoothing parameters depend on the unknown probability density function. 
In contrast to smoothing methods, shape constrained nonparametric density estimation is fully automatic and does not depend on the underlying probability distribution, though this comes at the expense of worse $L^1$ convergence rates for smooth densities~\cite{eggermont2000maximum}. Some previously studied classes of functions include non-increasing~\cite{grenander1956theory}, convex~\cite{groeneboom2001estimation}, $k$-monotone~\cite{balabdaoui2007estimation} and $s$-concave~\cite{MR3485950}. We refer the reader to~\cite{silverman1998density,scott2015multivariate, tsybakov2009introduction, groeneboom2014nonparametric} for general references on nonparametric statistics. The definitions of $k$-monotone and $s$-concave can be found in~\cite{balabdaoui2005estimation} and~\cite{ dharmadhikari1988unimodality}, respectively.

In this paper we focus on the class of log-concave densities, which is an important special case of $s$-concave densities. The choice of log-concavity is attractive for several reasons. First of all, most common univariate parametric families are log-concave, including the normal, Gamma with shape parameter greater than one, Beta densities with parameters greater than~1, Weibull with parameter greater than 1 and others. Furthermore, log-concavity is used in reliability theory, economics and political science~\cite{bagnoli2005log}. In addition to this, log-concave densities have several desirable statistical properties. For example, log-concavity implies unimodality but log-concave density estimation avoids the spiking phenomenon common in general unimodal estimation~\cite{duembgen2009}. Moreover, this class is closed under convolutions and taking pointwise limits~\cite{cule2010theoretical}. We refer the reader to~\cite{MR3881205} for an overview of the recent progress in the field.

Let $X=(x_1,x_2,\ldots,x_n)$ be a point configuration in $\mathbb{R}^d$ with weights $w=(w_1,w_2,\ldots,w_n)$ such that $w_i \geq 0$ and $w_1+w_2+\cdots+w_n=1$. The log-concave \emph{maximum likelihood estimation (MLE)} problem  aims to find a Lebesgue density that solves
\begin{equation} \label{log_concave_MLE_constrained}
\max \sum_{i=1}^n w_i \log(f(x_i)) \text{ s.t. } \log(f) \text{ is concave and } \int_{\mathbb{R}^d} f(x) dx = 1.
\end{equation}
It has been shown that the solution exists with probability $1$ and is unique, and its logarithm is a tent function, i.e.,~a piecewise linear function with regions of linearity inducing a subdivision of the convex hull of $X$~\cite{walther2002detecting, pal2007estimating, cule2010maximum, robeva2019geometry}, see Figure~\ref{figure: 14 point} for an example. While MLE is the most widely studied estimator in this setting, it is not the only one, for  examples see~\cite{pmlr-v65-diakonikolas17a, MR3899822}.

The maximum likelihood estimator is attractive because of its consistency under general assumptions~\cite{pal2007estimating,duembgen2009,cule2010theoretical, dumbgen2011approximation} and  superior performance compared to kernel-based methods with respect to mean integrated squared error, as observed in simulations~\cite{cule2010maximum}. At the same time, the convergence rate is still an open question and only lower~\cite{kim2016global, kur2019optimality} and upper~\cite{kim2016global, pmlr-v75-carpenter18a} bounds are known. Further theoretical properties have been studied for some special cases of log-concave densities, e.g.,~$k$-affine densities~\cite{kim2018adaptation} and totally positive densities~\cite{robeva2018maximum}. Several algorithms have been developed to compute the log-concave MLE in one dimension~\cite{rufibach2007computing} and in higher dimensions~\cite{cule2010maximum,axelrod2019efficient,rathke2018fmlogconc}. Software implementations include  \texttt{R} packages such as \texttt{logcondens}~\cite{logcondens} and \texttt{cnmlcd}~\cite{cnmlcd} in one dimension, and \texttt{LogConcDEAD}~\cite{LogConcDEAD} and \texttt{fmlogcondens}~\cite{rathke2018github} in higher dimensions. 

\begin{example}
\label{motivatingexample}
The starting point of this paper is the following problem. Consider the sample of $14$ points in $\mathbb{R}^2$ with uniform weights:
\begin{align*}
X=\left((0,1),(0,9),(1,4),(2,4),(2,6),(3,3),(5,5),(6,3),(6,9),(7,6),(7,8),(8,9),(9,5),(9,9)\right).
\end{align*}
How many cells does the subdivision induced by the logarithm of the optimal log-concave density  have?

Using the \texttt{R} package \texttt{LogConcDEAD} with default parameters, one obtains that the logarithm of the maximum likelihood estimate is a piecewise linear function with seven unique linear pieces. 
However, when one investigates the optimal density more closely, it appears that several linear pieces are similar. For example, a visual inspection of the optimal density depicted in Figure \ref{figure: 14 point} 
makes it impossible to distinguish all $7$ regions and
suggests that there are only four unique linear pieces.  Using \texttt{LogConcDEAD} one also obtains the two triangles, but according to the \texttt{LogConcDEAD} output the quadrangle consists of two linear pieces and the hexagon consists of three linear pieces. The subdivision corresponding to the \texttt{LogConcDEAD} result is depicted in Figure \ref{fig:7-cell}.  What is the true number of unique linear pieces of the optimal density? Is it four, seven or another value?

Theoretically, the algorithm used in \texttt{LogConcDEAD} finds the true optimal density, however, in practice, the answer is a numerical approximation. By changing the parameter \texttt{sigmatol} from default value $10^{-8}$ to $10^{-10}$, \texttt{LogConcDEAD} outputs four unique linear pieces, exactly as we observed in Figure \ref{figure: 14 point}. Although it might seem obvious that four is the correct number of linear pieces, in reality the situation is more complicated, see Example~\ref{finalexample}. How do we find the correct number of linear pieces?

\begin{figure}
\centering
\includegraphics[width=0.55\textwidth]{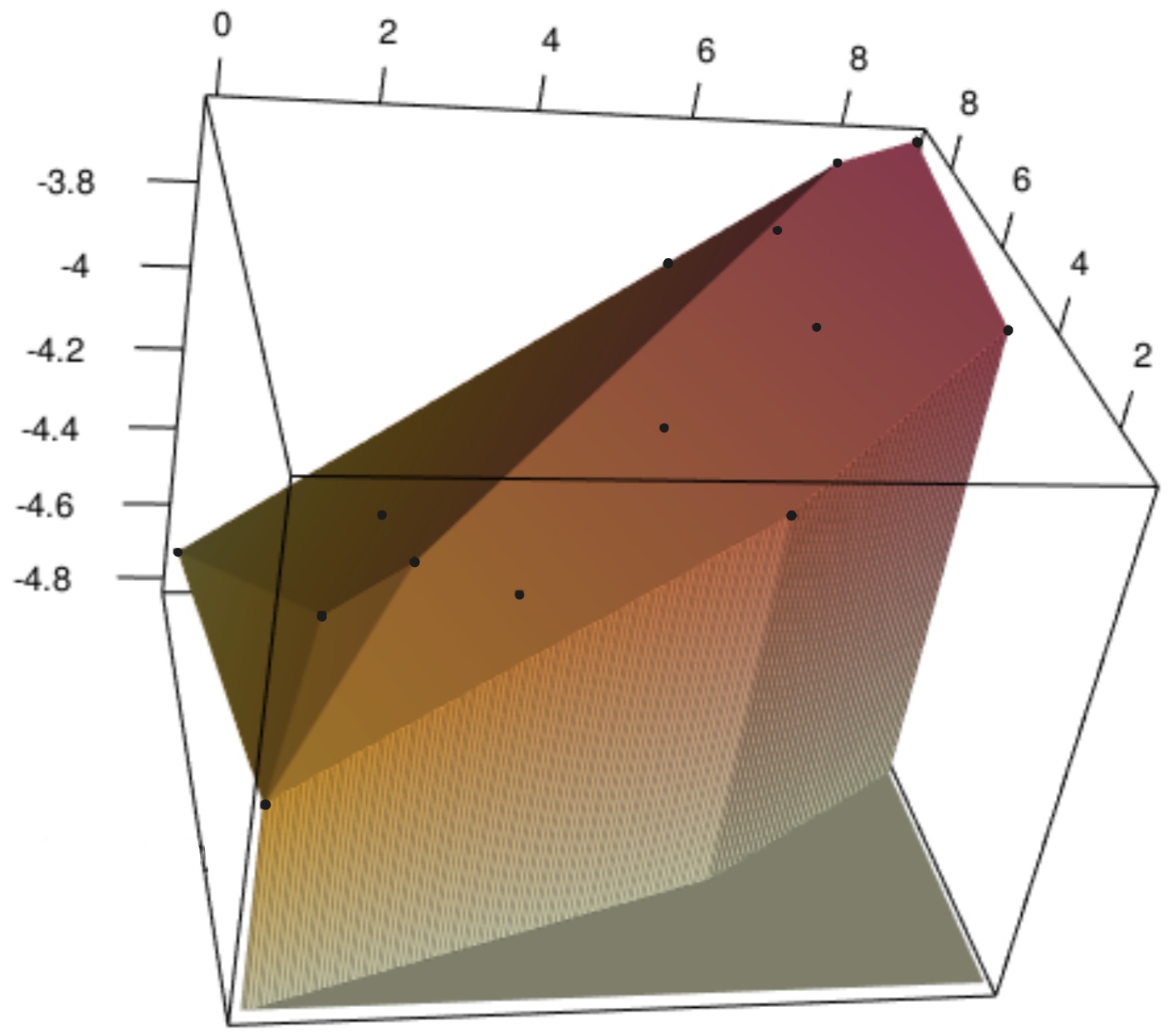}
\caption{The optimal tent function for the sample of 14 points in Example~\ref{motivatingexample}.}
\label{figure: 14 point}
\end{figure}
\end{example}

The goal of this paper is to study exact solutions to log-concave maximum likelihood estimation. An exact solution will have three different meanings in this paper. First, one might hope that it is an algebraic number. This would enable exact symbolic computations by way of storing a floating point approximation of a number along with a polynomial that vanishes on it. Such computations are not possible for transcendental numbers. Thus, the first main result of our paper is Theorem \ref{th:transcend}, which states that the heights at the sample points of the logarithm of the log-concave density estimate are transcendental for an open ball of weights. 

Second, in light of Theorem \ref{th:transcend}, we would like to express the maximum likelihood estimator in closed form using well-known mathematical operations and functions, although not necessarily elementary functions. 
In the simplest case of one cell in one dimension, we derive the log-concave density estimator in closed form using the generalized  $W$-Lambert function, see Proposition~\ref{lemma:one_dim_one_cell}. It is known that the generalized  $W$-Lambert function is not an elementary function.  More generally, solving the MLE can be restated as a collection of polynomial-exponential systems of equations, which have been studied in the literature. However, even in the case of two equations, only bounds on solutions are known~\cite{Maignan:1998}. This suggests that it might be difficult to express the log-concave maximum likelihood estimator in closed form. As an alternative,
we turn to Smale's $\alpha$-theory, which we describe briefly now.

Third, given a sufficiently close floating point solution to the MLE problem, one hopes that it can be refined to any desired precision using Newton iteration or other techniques. A natural question arises: when is the approximate solution good enough for these methods to succeed? A way to make this mathematically rigorous is Smale's $\alpha$-theory \cite{BCSS1998, S1986}, which we discuss in Section \ref{section:Smale-alpha-certify}. We obtain the $\alpha$-certified solutions to log-concave density estimation. This allows us to test and compare numerical solvers, as well as rigorously decide the certified, correct subdivision for a given log-concave density estimation problem. Our methods are especially relevant when the precision of the log-concave density estimate is important. This opens new pathways to answering the motivating question: what is the correct number of cells?

The code for computations in this paper can be found at~\cite{githubrepo}.

\section{Geometry of log-concave maximum likelihood estimation}\label{sect:LogConcMLE}

We start by reviewing the geometry of log-concave maximum likelihood estimation mostly following~\cite{robeva2019geometry}. 

\begin{definition}
Let $P$ be the convex hull of a point configuration $X=(x_1,x_2,\ldots,x_n)\subset \mathbb{R}^d$. For a fixed real vector $y \in \mathbb{R}^n$,  
we define a function $h_{X,y}$ on $\mathbb{R}^d$, called the \emph{tent function}, as the smallest concave function such that $h_{X,y}(x_i) \geq y_i$ for $i=1,\ldots,n$. Here the term smallest means that for any other concave function $\bar{h}$ on $\R^d$ such that $\bar{h}(x_i) \geq y_i$ for $i=1,\ldots,n$, one must have $\bar{h}(x) \geq h_{X,y}(x)$ for all $x \in \R^d$. The tent function $h_{X,y}$ is piecewise linear on $P$ with linear pieces equal to upper facets of the convex hull of the points $(x_1,y_1),(x_2,y_2),\ldots,(x_n,y_n)$ in $\mathbb{R}^{d+1}$. 
We have $h_{X,y}(x) = -\infty$ at all points $x \in \mathbb{R}^d$ outside~$P$. If $h_{X,y}(x_i) = y_i$ for $i=1,\ldots,n$, then $y$ is called \emph{relevant}.
\end{definition}

It was shown by Cule, Samworth and Stewart for uniform weights~\cite{cule2010maximum} and by Robeva, Sturmfels and Uhler in general~\cite{robeva2019geometry} that the constrained optimization problem~(\ref{log_concave_MLE_constrained}) of finding the log-concave maximum likelihood estimate is equivalent to the unconstrained optimization problem
\begin{equation} \label{log_concave_MLE_unconstrained}
\max_{y \in \mathbb{R}^n} \,\,\,\, w \cdot y - \int_P \exp \big( h_{X,y}(t)\big)dt.
\end{equation}
Moreover, the log-concave maximum likelihood estimate is a tent
function with tent poles at some of the $x_i$. Therefore finding the
log-concave density which maximizes the likelihood of  $(X,w)$ is equivalent to finding an optimal height vector $y^*$. 

\begin{definition} \label{def: reg triangulation}
We follow the definitions in \cite{MR2743368}. Given a point configuration $X$ in $\R^d$, a \emph{subdivision} $\Delta$ of $X$ is a collection of $d$-polytopes, denoted $\sigma_i$, such that the union of polytopes in $\Delta$ equals $\conv(X)$, the vertex set of polytopes in $\Delta$ is contained in $X$ and the intersection of polytopes in $\Delta$ can only happen along lower dimensional faces. A subdivision $\Delta$ is called a \emph{triangulation}, if all polytopes in $\Delta$ are simplices.  A triangulation $\Delta$ of the point configuration $X$ is called \emph{maximal}, if every element of $X$ is a vertex of a simplex in $\Delta$.  A subdivision is called \emph{regular} if its full dimensional cells $\sigma_i$ are combinatorially equivalent to the regions of linearity of a tent function on $X$ for some height vector $y\in \R^n.$
\end{definition}

\begin{cor}\cite[Corollary 2.6]{robeva2019geometry}\label{cor:optsecondary}
To find the optimal height vector $y^*$ in~(\ref{log_concave_MLE_unconstrained}) is to maximize the following rational-exponential objective function over $y \in \mathbb{R}^n$:
\begin{equation}
\label{eq:optsecondary}
S(y_1,\dots,y_n) = \,\,\,w \cdot y \,- \,\sum_{\sigma \in \Delta}
\sum_{i \in \sigma}
\frac{{\rm vol}(\sigma) \cdot {\rm exp}(y_i)} 
{\prod_{\alpha \in \sigma \backslash i} (y_i-y_{\alpha})},
\end{equation}
where $\Delta$ is any regular triangulation that refines the regular subdivision induced by the tent function $h_{X,y}$.
\end{cor}

If $y$ induces a regular subdivision $\Delta$ that is not a maximal regular triangulation, then we can consider any maximal regular triangulation that refines $\Delta$. Thus if there are $m$ maximal regular triangulations of $X$, then to find the optimal $y^*$ we must compare the optimal values $y^*_{\Delta_1}, y^*_{\Delta_2}, \dots, y^*_{\Delta_m}$ which are obtained by solving the optimization problem (\ref{eq:optsecondary}) $m$ times, once for each maximal regular triangulation $\Delta_1, \Delta_2, \dots, \Delta_m$.

\begin{notation}\label{notation: S restriction}
We will denote by $S_{\Delta}$ the function given by the right hand side of~(\ref{eq:optsecondary}) for a fixed triangulation~$\Delta$.  
\end{notation}

\begin{example} \label{ex:2-5-7}
Fix $d=1$, $n=3$ and $X=(2,5,7)$. 
The configuration $X$ has two triangulations 
$\Delta_1 = \{ \{1,3\}\}$ and 
$\Delta_2 = \{\{ 1,2\}, \{2,3\} \}$, which are both regular triangulations. Only $\Delta_2$ is a maximal triangulation. Hence solving the optimization problem~(\ref{log_concave_MLE_unconstrained}) is equivalent to maximizing the objective function
\begin{equation}
\label{eq 2-cells}
     S_{\Delta_2}=w \cdot y - 3 \frac{e^{y_1} - e^{y_2}}{y_1-y_2} -2 \frac{e^{y_2} - e^{y_3}}{y_2-y_3}.
\end{equation}
If $y_1=y_2$ or $y_2=y_3$, then a denominator on the right hand side of~(\ref{eq 2-cells}) becomes zero. However, the objective function in the formulation~(\ref{log_concave_MLE_unconstrained}) can be still simplified to
\begin{equation*}
w \cdot y -3e^{y_2}-2 \frac{e^{y_2} - e^{y_3}}{y_2-y_3}\quad \text{or} \quad w \cdot y - 3 \frac{ (e^{y_1} - e^{y_2})}{y_1 - y_2}-2e^{y_2}.
\end{equation*}

To visualize the situation, we consider the Samworth body 
\[\sS(X)=\left\{y \in \R^3:\int_{P}\exp(h_{X,y}(t))dt\leq 1\right\},\]
which was introduced in~\cite{robeva2019geometry}.
The unconstrained optimization problem~(\ref{log_concave_MLE_unconstrained}) is equivalent to the constrained optimization problem of maximizing the linear function $w \cdot y$ over the Samworth body.  
For different choices of weight vector $w=(w_1,w_2,w_3)$, we obtain different optimal height vectors $y=(y_1,y_2,y_3)$ on the surface of the Samworth body, and the height vector determines the triangulation. The Samworth body consists of two regions that can be seen in Figure~\ref{fig:Samworth-body-for-three-points}. The green region comes from the one-simplex triangulation ${\Delta_1=\{ \{1,3\}\}}$, while the red region comes from the two-simplex triangulation ${\Delta_2 = \{\{ 1,2\}, \{2,3\} \}}$. Moreover, one can see lines separating the green region into two pieces and the red region into three pieces (ignore the curve separating the green and the red regions for now). These lines correspond to the degenerate cases where $y_1 = y_3$, $y_1 = y_2$ or $y_2 = y_3$, and hence the right hand side of~(\ref{eq:optsecondary}) is not defined. Therefore those lines are simply artifacts of the reformulation~(\ref{eq:optsecondary}) since in the original unconstrained setting~(\ref{log_concave_MLE_unconstrained}) these points present no difficulty.  The intersection of the three lines is the point $(-\log 5,-\log 5,-\log 5)$.

\begin{figure}
    \centering
    \includegraphics[width = 0.7\textwidth]{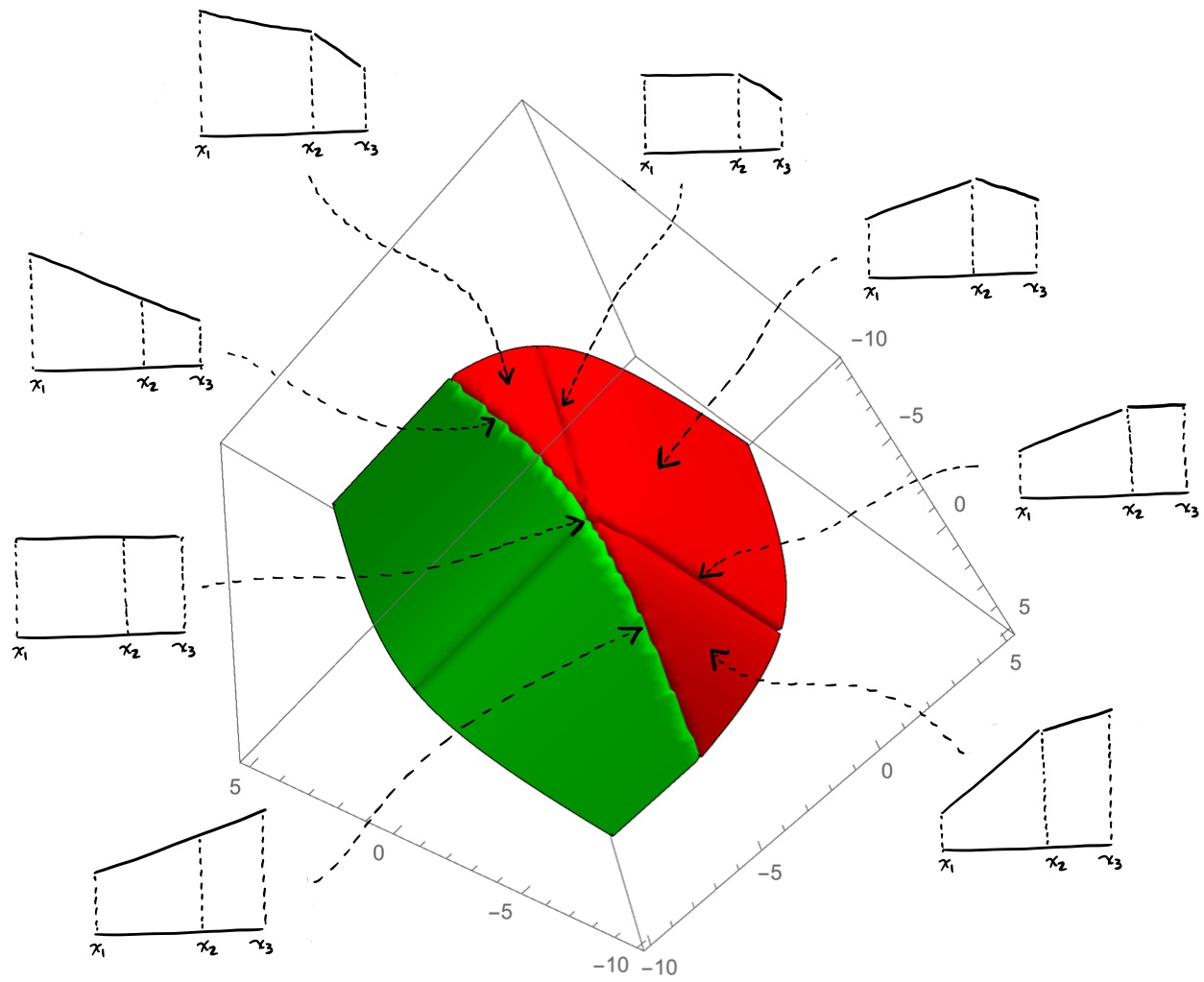}
    \caption{The Samworth body for $X=(2,5,7)$.}
    \label{fig:Samworth-body-for-three-points}
\end{figure}

Consider the curve separating the green and red regions of the Samworth body. 
This curve is made of all the  points $y$ that form a relevant
  tent function, inducing the subdivision $\Delta_1$.  To understand the green region, see the piecewise linear functions drawn in Figure \ref{fig:varying-y2}. Since the lowest (dotted) function is not concave, it is invalid as a tent function.  Therefore, if the height $y_2$ is too low, the optimal tent function will be the (solid-line) linear function. 
In effect, the optimal tent-function ignores heights $y_i$ if they are too low. This basic phenomenon is responsible for the green part of the Samworth body being flat in the $y_2$ direction, meaning that it is a pencil of half-lines parallel to the $y_2$-axis.

\begin{figure}
    \centering
    \includegraphics[width=0.95\textwidth]{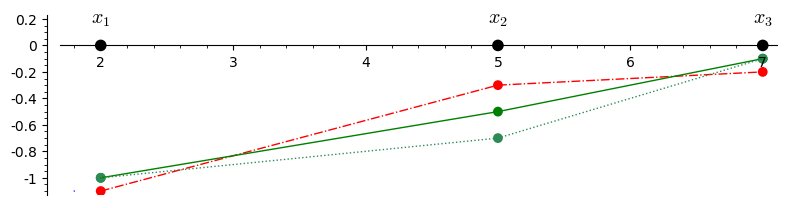}
    \caption{The red tent function corresponds to a vector $y$ in the red region of the Samworth body. The solid green tent function corresponds to a vector $y$ on the curve separating red and green regions of the Samworth body. The dotted green function is not convex. Its height vector $y$ belongs to the green region of the Samworth body and both green sets of heights give the same tent function.}
    \label{fig:varying-y2}
\end{figure}

The transition from the red region to the green region is not smooth. For every $y$ on the curve between the green and red regions, there is a two-dimensional cone of weight vectors that give $y$ as an optimal solution. The generators of this cone are described in~\cite[Theorem~3.7]{robeva2019geometry}. The optimal height vector $y^*$ for $w = (\frac{1}{3},\frac{1}{3},\frac{1}{3})$ lies on the curve between the red and green regions. It is not a critical point of the function~(\ref{eq 2-cells}), because $w$ is not a normal vector to the red region at the point $y^*$.
\end{example}

We now return to the general situation and consider the specific approach of critical equations for solving the optimization problem~(\ref{eq:optsecondary}). Let $X = (x_1,\dots,x_n)$ be a configuration of $n$ points $x_i \in \mathbb{R}^d$. Fixing a maximal regular triangulation $\Delta$ of our point configuration $X$, we can find the optimal $y^*_{\Delta}$  for $S_{\Delta}$ in~(\ref{eq:optsecondary}) over $y \in \mathbb{R}^n$ by solving the system of critical equations $\partial S_{\Delta}/\partial y_i = 0$. These partial derivatives take the form (see \cite[Proof of Lemma 3.4]{robeva2019geometry}):
\begin{align}
\frac{\partial S_{\Delta}}{\partial y_i} \,\,\,= \,\,\,w_i \,\, - \,
&\sum_{\substack{\sigma\in\Delta,\\ i\in\sigma}}\vol(\sigma)\exp(y_i)\frac1{\prod_{\alpha\in \sigma\setminus i }(y_i-y_{\alpha})}\left(1 - \sum_{\alpha\in \sigma\setminus i }\frac{1}{(y_i - y_{\alpha})}\right)\nonumber \\
-&\, \sum_{\substack{\sigma\in\Delta \\ i\in\sigma}}\text{vol}(\sigma)\sum_{j\in \sigma\setminus i } \exp(y_j) \frac{1}{\prod_{\alpha\in \sigma\setminus j} (y_j - y_{\alpha})} \frac{1}{(y_j - y_i)}.
\label{critical-equations}
\end{align}
\begin{definition}\label{def: score equation matrix}
For a fixed maximal regular triangulation $\Delta$ of $X$, let $A$ be the matrix such that the system of $n$ critical equations (\ref{critical-equations}) can be written in the form
\begin{equation}\label{equation:score_equations_matrix_form}
  A e^{y} = w,
\end{equation}
where $e^{y}$ is a column vector of exponentials $(e^{y_1}, e^{y_2}, \dots, e^{y_n})^T$,  and $w$ is a column vector of weights $(w_1,\dots,w_n)^T$.  The matrix $A$ is called the \emph{score equation matrix}. 
\end{definition}
The entries of $A$ are in the field of rational functions in the variables $y_1,\dots,y_n$. Diagonal entries of $A$ are
$$
A_{j,j} = \sum_{\substack{\sigma\in\Delta,\\ j\in\sigma}} \vol (\sigma) \frac{1}{\prod_{\alpha \in \sigma \backslash j} (y_j-y_{\alpha})} \left( 1 - \sum_{\alpha \in \sigma \backslash j} \frac{1}{(y_j - y_{\alpha})} \right)
$$
and off-diagonal entries of $A$ are
$$
A_{i,j} = \sum_{\substack{\sigma\in\Delta,\\ i,j\in\sigma}} \vol(\sigma) \frac{1}{\prod_{\alpha \in \sigma \backslash j} (y_j - y_{\alpha})} \frac{1}{(y_j - y_i)}.
$$
The matrix $A$ can be written as a sum of matrices over maximal simplices $\sigma \in \Delta$. This will be described explicitly in the proof of Theorem~\ref{theorem: invertibility of A}. 

There are two caveats when solving the optimization problem~(\ref{eq:optsecondary}) using the method of critical equations. First, it is not enough to consider the system of critical equations $\partial S_{\Delta}/\partial y_i = 0$ only for each of the maximal regular triangulations $\Delta$, since the optimization problem~(\ref{eq:optsecondary}) is not smooth. One has to consider a system of critical equations for each subdivision of $X$. For a general subdivision~$\Delta$ of $X$, this system is constructed in the following way. We consider $S_{\Delta'}(y_1,\dots,y_n)$ for any maximal triangulation $\Delta'$ that refines $\Delta$, substitute $y_i$ that can be expressed in terms of other $y$'s in the subdivision $\Delta$ and construct the system of critical equations $\partial \widetilde{S}_{\Delta}/\partial y_i = 0$ for the resulting function $\widetilde{S}_{\Delta}$. For maximal triangulations, we have $\widetilde{S}_{\Delta}=S_{\Delta}$ and the system of critical equations is given by~(\ref{critical-equations}). We will demonstrate this phenomenon on the point configuration from Example~\ref{ex:2-5-7}. 

\begin{example} \label{ex:2-5-7-critical-equations}
Recall that $d=1$, $n=3$ and $X=(2,5,7)$. 
The configuration $X$ has two triangulations 
$\Delta_1 = \{ \{1,3\}\}$ and 
$\Delta_2 = \{\{ 1,2\}, \{2,3\} \}$. Let $w = (\frac{1}{3},\frac{1}{3},\frac{1}{3})$. The output from \texttt{LogConcDEAD} suggests that the optimal tent function is supported on one cell, with heights given by $y^*_1=-1.816665$, $y^*_2=-1.576024$ and ${y^*_3=-1.415597}$. However, the vector $y^*$ is neither a critical point of $S_{\Delta_2}$ nor of the function
$$
S_{\Delta_1}=w \cdot y - 5 \frac{e^{y_1} - e^{y_3}}{y_1-y_3}.
$$

This can be seen by taking partial derivatives of these functions with respect to $y_1,y_2,y_3$ and substituting $y^*_1,y^*_2,y^*_3$. In the case of $\partial S_{\Delta_1}/\partial y_i = 0$, it is particularly easy to see that there are no solutions, since $\partial S_{\Delta_1}/\partial y_2 = w_2 \neq 0$. In the case of $\partial S_{\Delta_2}/\partial y_i = 0$, the system of critical equations fails to certify in the sense of Section~\ref{section:Smale-alpha-certify}.

The points $(x_1,y^*_1),(x_2,y^*_2),(x_3,y^*_3)$ being collinear is equivalent to
$(x_2,y^*_2) = \lambda_1 (x_1,y^*_1) + \lambda_3 (x_3,y^*_3)$
where $\lambda_1,\lambda_3 \geq 0$, $\lambda_1+\lambda_3 = 1$. Since $x_1 = 2, x_2 = 5, x_3 = 7$, we have $\lambda_1 = \frac{2}{5}, \lambda_3 = \frac{3}{5}$. Hence $y_2 = \frac{2}{5}y_1 + \frac{3}{5} y_3.$ Substituting this expression into the objective function~(\ref{eq 2-cells}) we get 
\begin{equation*}
    \widetilde{S}_{\Delta_2} = \left( w_1 + \frac{2}{5} w_2\right)y_1 + \left( w_3 + \frac{3}{5} w_2 \right)y_3 - 5 \frac{e^{y_1} - e^{y_3}}{y_1-y_3}
\end{equation*}
which for uniform weights $w = (\frac{1}{3},\frac{1}{3},\frac{1}{3})$ becomes
\begin{equation}\label{equation:restriction-to-CtildeDelta}
    \widetilde{S}_{\Delta_2} = \frac{7}{15} y_1 + \frac{8}{15} y_3 - 5 \frac{e^{y_1} - e^{y_3}}{y_1-y_3}.
\end{equation}
We will verify in Example~\ref{exampleagainst} that $y^*$ is a critical point of the function $\widetilde{S}_{\Delta_2}.$
\end{example}

The second caveat is that to find the optimal tent function, it is not enough to merely compare the optimal critical points $y^*_{\Delta}$ of $\partial S_{\Delta}/\partial y_i = 0$ for each subdivision $\Delta$. Denote by $Y_{\Delta}$ the set of $y$ that induce a subdivision that is equal to or coarser than $\Delta$. For each $\Delta$, it also has to be checked that $y^*_{\Delta}$ is in $Y_{\Delta}$. Thus if $y^*_{\Delta}$ is not in $Y_{\Delta}$, then $y^*_{\Delta}$ should be discarded. If the maximum of $S_{\Delta}$ over $Y_{\Delta}$ is not a critical point of $S_{\Delta}$, then the maximum must be on the boundary of $Y_{\Delta}$, see Figure \ref{fig:restricting-to-CDelta} for an illustration. The boundary of $Y_\Delta$ is stratified into regions $Y_{\widetilde{\Delta}}$ corresponding to the various subdivisions $\widetilde{\Delta}$ which are refined by  $\Delta$.  Hence one should consider critical points for strictly coarser subdivisions $\widetilde{\Delta}$. 

\begin{figure}
    \centering
    \includegraphics[width=0.5\textwidth]{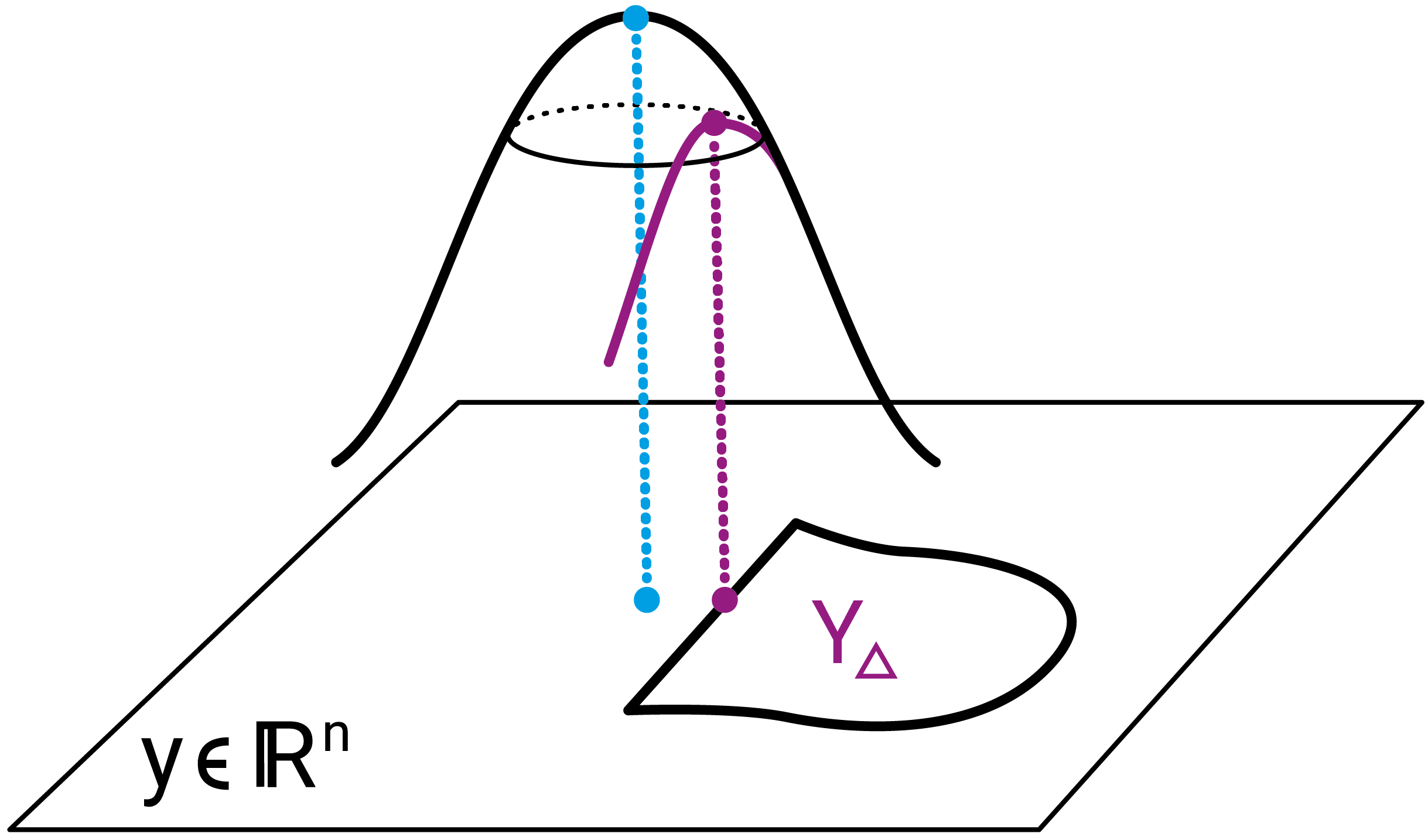}
    \caption{Maximizing $S_{\Delta}$ over $y$ restricted to $Y_\Delta$.}
    \label{fig:restricting-to-CDelta}
\end{figure}

\begin{example} \label{example:optimal-solutions-for-all-subdivisions}
We consider the point configuration $X=\{0,1,2,3,4\} \subseteq \mathbb{R}$ and the weight vector $w=(3/15,4/15,5/15,2/15,1/15)$. This point configuration has exactly eight subdivisions. For each subdivision $\Delta$, we use the  \texttt{Mathematica} commmand \texttt{NMaximize} to find the maximum $y^*_{\Delta}$ of the function $S_{\Delta}$. For each subdivision $\Delta$, the smallest piecewise-linear function $f^*_{\Delta}$ such that $f^*_{\Delta}(x_i)\geq y^*_{\Delta,i}$ for $i=1,\ldots,5$ is depicted in Figure~\ref{fig:five-points-optimal-y}. We have $\int_P \exp(f^*_{\Delta}(t))dt=1$ for all subdivisions $\Delta$. This implies that if $y^*_{\Delta}$ is not relevant, then $\exp(h_{X,y^*_{\Delta}})$ is not a distribution.

\begin{figure}
\centering
\begin{subfigure}{0.22\textwidth}
    \includegraphics[width=\textwidth]{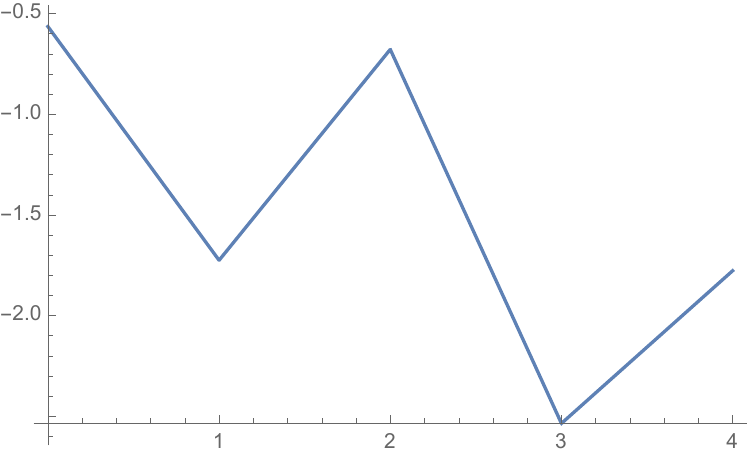}
    \caption{$\{12,23,34,45\}$}\label{fig:subdivision1}
\end{subfigure}
\hfill
\begin{subfigure}{0.22\textwidth}
    \includegraphics[width=\textwidth]{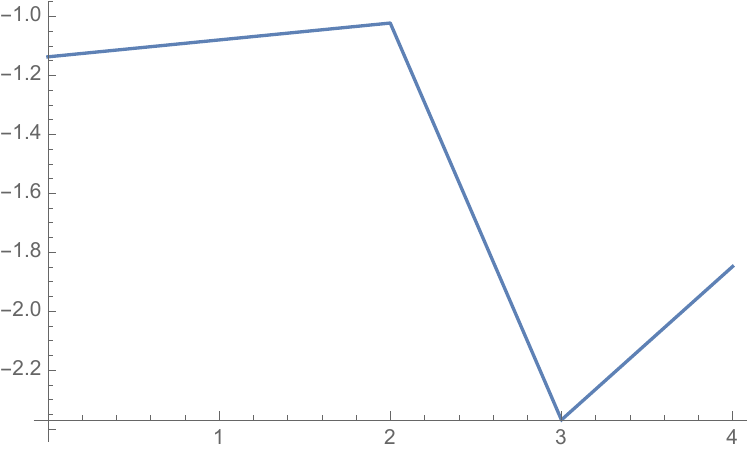}
    \caption{$\{13,34,45\}$}\label{fig:subdivision2}
\end{subfigure}
\hfill
\begin{subfigure}{0.22\textwidth}
    \includegraphics[width=\textwidth]{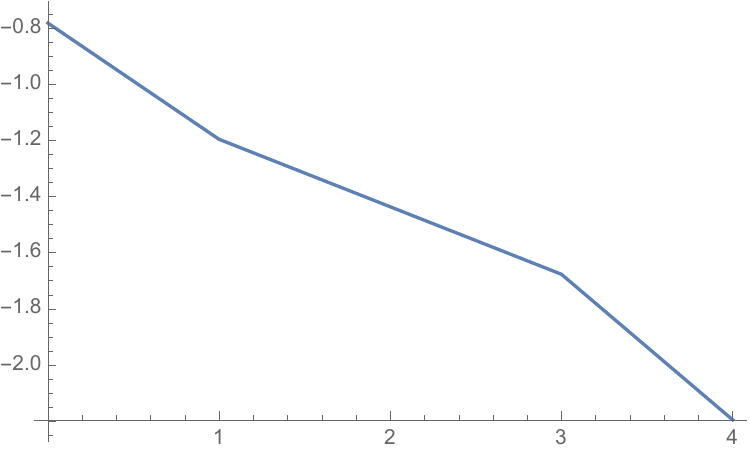}
    \caption{$\{12,24,45\}$}\label{fig:subdivision3}
\end{subfigure}
\hfill
\begin{subfigure}{0.22\textwidth}
    \includegraphics[width=\textwidth]{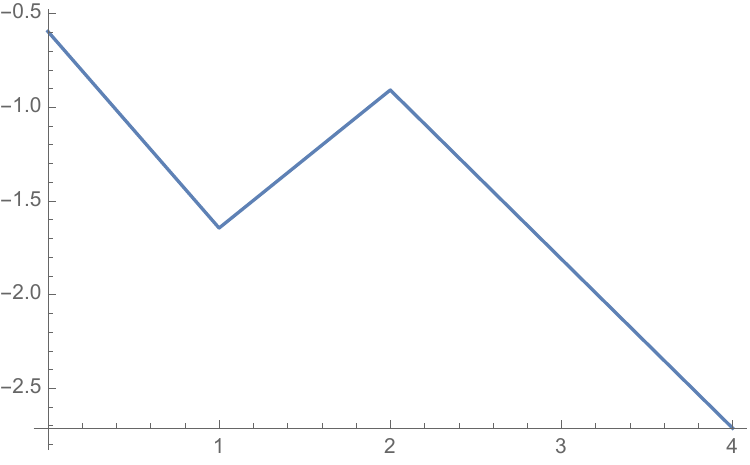}
    \caption{$\{12,23,35\}$}\label{fig:subdivision4}
\end{subfigure}

\begin{subfigure}{0.22\textwidth}
    \includegraphics[width=\textwidth]{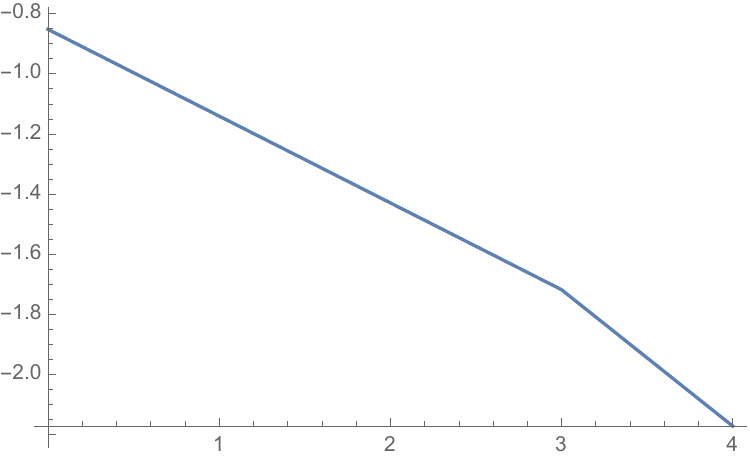}
    \caption{$\{14,45\}$}\label{fig:subdivision5}
\end{subfigure}
\hfill
\begin{subfigure}{0.22\textwidth}
    \includegraphics[width=\textwidth]{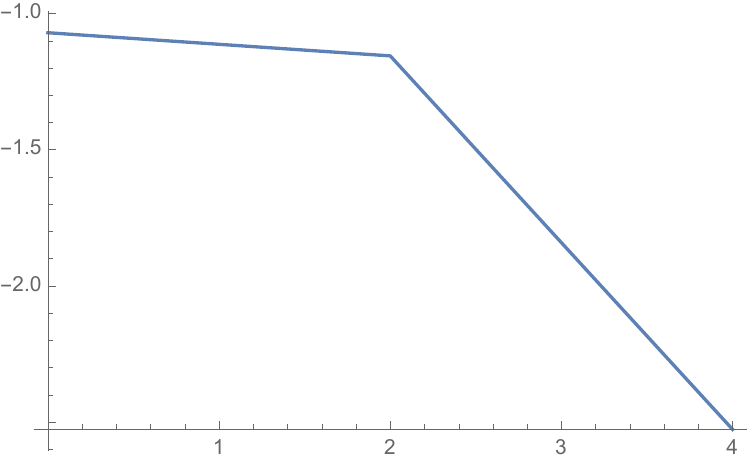}
    \caption{$\{13,35\}$}\label{fig:subdivision6}
\end{subfigure}
\hfill
\begin{subfigure}{0.22\textwidth}
    \includegraphics[width=\textwidth]{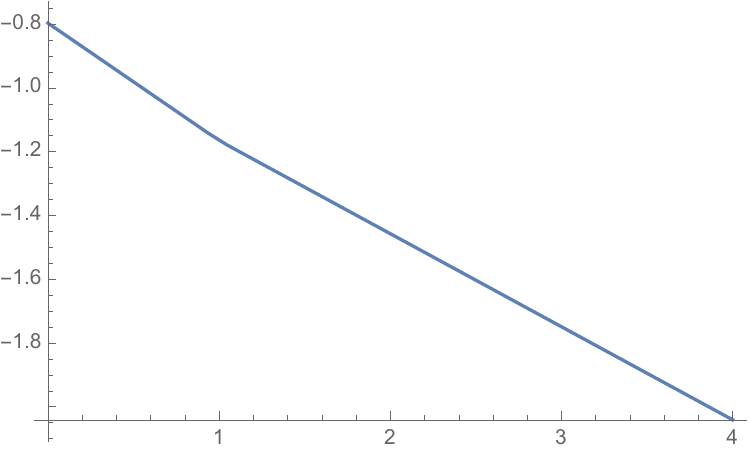}
    \caption{$\{12,25\}$}\label{fig:subdivision7}
\end{subfigure}
\hfill
\begin{subfigure}{0.22\textwidth}
    \includegraphics[width=\textwidth]{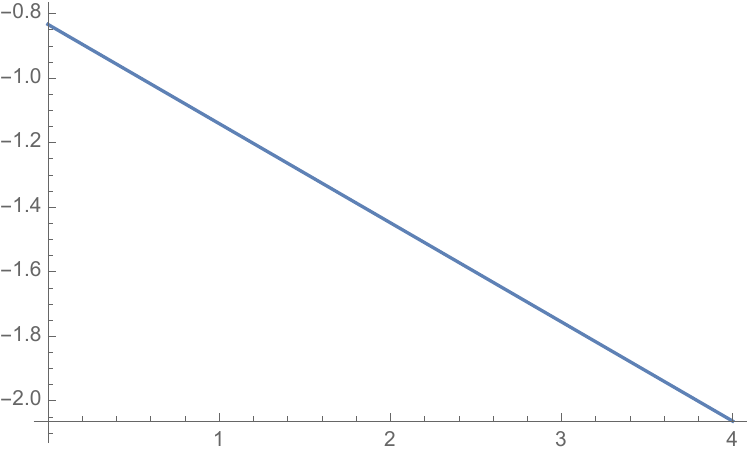}
    \caption{$\{15\}$}\label{fig:subdivision8}
\end{subfigure}
   \hfill
\caption{Piecewise-linear functions induced by $y^*_{\Delta}$ maximizing $S_{\Delta}$ for each subdivision $\Delta$ in Example~\ref{example:optimal-solutions-for-all-subdivisions}. The notation $ij$ in the subcaptions refers to the set $\{i,j\}$.}
\label{fig:five-points-optimal-y}
\end{figure}

The optimal tent function is supported on the subdivision $\{\{1,3\},\{3,5\}\}$. Also subdivisions $\{\{1,4\},\{4,5\}\}$ and $\{\{1,5\}\}$ give concave piecewise-linear functions $f^*_{\Delta}$, however, the value of $S_{\Delta}$ at $y^*_{\Delta}$ is less for these subdivisions (respectively $-2.32524$ and $-2.32556$) than for the optimal subdvision ($-2.31007$). Moreover, only for the optimal subdivision we obtain $y^*_{\Delta}$ that is close to the optimal $y^*_{\Delta}$ obtained by \texttt{LogConcDEAD}. In this example, \texttt{LogConcDEAD} gives $y^*_{1}=-1.070377$. For the eight subdivisions in Figure~\ref{fig:five-points-optimal-y}, we get the following values for the first coordinate of $y^*_{\Delta}$ using \texttt{Mathematica}: (a)  $-0.564769$ (b) $-1.13722$ (c) $-0.783036$ (d) $-0.595576$ (e) $-0.852468$ (f) $-1.07045$ (g) $-0.797148$ (h) $-0.833582$. Similarly for other coordinates of $y^*$, only $y^*_{\{\{1,3\},\{3,5\}\}}$ agrees with $y^*$ when rounded to the third decimal digit. This suggests a method for checking whether a subdivision supports the optimal tent function: The piecewise-linear function $f^*_{\Delta}$ should be concave and the height vector $y^*_{\Delta}$ should be close to $y^*$ obtained by \texttt{LogConcDEAD}.

We see from this example, if a subdivision $\Delta$ is incompatible with the optimal subdivision, then $f^*_{\Delta}$ might or might not be concave. The subdivisions $\{\{1,4\},\{4,5\}\}$ and $\{\{1,2\},\{2,5\}\}$ are both incompatible with the subdivision $\{\{1,3\},\{3,5\}\}$, and  $f^*_{\{\{1,4\},\{4,5\}\}}$ is concave whereas $f^*_{\{\{1,2\},\{2,5\}\}}$ is not concave. In all examples that we have done, if a subdivision $\Delta$ refines the optimal subdivision, then $f^*_{\Delta}$ is not concave and if a subdivision $\Delta$ is coarser than the optimal subdivision, then $f^*_{\Delta}$ is concave. Whether this is true in general, is left as an open question.
\end{example}

\section{Transcendentality and closed-form solutions}\label{section:one-dimension-d-equals-1}
In this section we use notions from geometric combinatorics to study the structure of (\ref{def: score equation matrix}). In particular, we will prove that the matrix $A$ is invertible. This will be our main tool in proving the transcendentality of log-concave MLE and deriving closed form solutions  in the one-dimensional one cell case using Lambert functions.
\subsection{Score equation matrix invertibility and transcendentality}

Towards proving transcendentality, we first investigate the invertibility of the matrix $A$.   

\begin{theorem}\label{theorem: invertibility of A}
Consider a point configuration $X=\left(x_1,\ldots, x_n\right)$ in $\R^d$, let $\Delta=\{\sigma_1,\ldots,\sigma_m\}$ be a
  maximal regular triangulation of $X$. The score equation matrix $A$ from~(\ref{equation:score_equations_matrix_form}) is invertible.
\end{theorem}

\begin{definition}\label{def: neighborhood}
Given a triangulation $\Delta$, we define the \emph{neighborhood} $\sN(j)$ of a vertex $j$ in $\Delta$ to be the set of vertices $$
\sN(j)=\left\{i: (i,j) \in \sigma_k \text{ for some }k\right\}.
$$
\end{definition}
Before giving the proof of Theorem~\ref{theorem: invertibility of A}, we illustrate the construction in the proof with a small example. 

\begin{example}
   Let $X=(x_1,x_2,x_3,x_4)$ be a four point configuration
   in $\R^2$ with $\Delta=\{\sigma_1,\sigma_2\}$, where
   $\sigma_1=\{1,2,3\}$ and $\sigma_2=\{2,3,4\}$. Let $A$ be the score equation matrix for the entire
   regular triangulation $\Delta$. Let us denote the
   difference $y_i-y_j $ by $y_{ij}$. Then $A=
   A(\sigma_1)+A(\sigma_2)$, where
    {\small
    \begin{align*}
    \frac{A(\sigma_1)}{\vol(\sigma_1)} &=
    \left[ \begin {array}{cccc} {\frac {1}{y_{{12}}y_{{13}}}}-{\frac {1
}{{y_{{12}}}^{2}y_{{13}}}}-{\frac {1}{y_{{12}}{y_{{13}}}^{2}}}&{
\frac {1}{{y_{{21}}}^{2}y_{{23}}}}&{\frac {1}{{y_{{31}}}^{2}y_{{32}}}
}&0\\ \noalign{\medskip}{\frac {1}{{y_{{12}}}^{2}y_{{13}}}}&{\frac {
1}{y_{{21}}y_{{23}}}}-{\frac {1}{{y_{{21}}}^{2}y_{{23}}}}-{\frac {1}{y_{{21}}{y_{{23}}}^{2}}}&{\frac {1}{y_{{31}}{y_{{32}}}^{2}}}&0\\ 
\noalign{\medskip}{\frac {1}{y_{{12}}{y_{{13}}}^{2}}}&{\frac {1}{
y_{{21}}{y_{{23}}}^{2}}}&{\frac {1}{y_{{31}}y_{{32}}}}-{\frac {1}{{y_{{31}}}^{2}y_{{32}}}}-{\frac {1}{y_{{31}}{y_{{32}}}^{2}}}&0
\\ \noalign{\medskip}0&0&0&0\end {array} \right],\\
\frac{ A(\sigma_2)}{\vol(\sigma_2)}&=
\left[ \begin {array}{cccc} 0&0&0&0\\ \noalign{\medskip}0&{\frac {1
}{y_{{23}}y_{{24}}}}-{\frac {1}{{y_{{23}}}^{2}y_{{24}}}}-{\frac {1}{y_
{{23}}{y_{{24}}}^{2}}}&{\frac {1}{{y_{{32}}}^{2}y_{{34}}}}&{\frac {1
}{{y_{{42}}}^{2}y_{{43}}}}\\ \noalign{\medskip}0&{\frac {1}{{y_{{23}}
}^{2}y_{{24}}}}&{\frac {1}{y_{{32}}y_{{34}}}}-{\frac {1}{{y_{{32}}}^{2}y_{{34}}}}-{\frac{1}{y_{{32}}{y_{{34}}}^{2}}}&{\frac {1}{y_{{42}}{
y_{{43}}}^{2}}}\\ \noalign{\medskip}0&{\frac {1}{y_{{23}}{y_{{24}}}^{
2}}}&{\frac {1}{y_{{32}}{y_{{34}}}^{2}}}&{\frac {1}{y_{{42}}y_{{43}}
}}-{\frac {1}{{y_{{42}}}^{2}y_{{43}}}}-{\frac {1}{y_{{42}}{y_{{43}}}^{2}}}\end {array} \right].
\end{align*}}
We define matrix $B$ to be the matrix $A$ with its $j$-th column multiplied by $\prod_{i\in \sN(j)}y_{ji}^2$, for all $j$ from $1$ to $4$. We obtain the following matrices
{\scriptsize
\begin{align*}
  \frac{B(\sigma_1)}{\vol(\sigma_1)} 
  &= \left[ \begin {array}{cccc} y_{{13}}y_{{12}}-y_{{12}}-y_{{13}}&{y_{{24}}}^{2}y_{{23}}&{y_{{34}}}^{2}y_{{32}}&0\\ \noalign{\medskip}y_{{13}}&y_{{21}}y_{{23}}{y_{{24}}}^{2}-{y_{{24}}}^{2}y_{{21}}-y_{{23}}{y_{{24}}}^{2}&{y_{{34}}}^{2}y_{{31}}&0\\
  \noalign{\medskip}y_{{12}}&{y_{{24}}}^{2}y_{{21}}&y_{{31}}y_{{32}}{y_{{34}}}^{2}-{y_{{34}}}^{2}y_{{31}}-{y_{{34}}}^{2}y_{{32}}&0\\
  \noalign{\medskip}0&0&0&0\end {array}\right],\\
 \frac{ B(\sigma_2)}{\vol(\sigma_2)} 
  &= \left[ \begin {array}{cccc} 0&0&0&0\\\noalign{\medskip}0&{y_{{21}}}^{2}y_{{23}}y_{{24}}-{y_{{21}}}^{2}y_{{23}}-{y_{{21}}}^{2}y_{{24}}&{y_{{31}}}^{2}y_{{34}}&y_{{43}}\\
  \noalign{\medskip}0&{y_{{21}}}^{2}y_{{24}}&y_{{32}}{y_{{31}}}^{2}y_{{34}}-{y_{{31}}}^{2}y_{{32}}-{y_{{31}}}^{2}y_{{34}}&y_{{42}}\\
  \noalign{\medskip}0&{y_{{21}}}^{2}y_{{23}}&{y_{{31}}}^{2}y_{{32}}&y_{{43}}y_{{42}}-y_{{42}}-y_{{43}}\end {array}
 \right]. 
\end{align*}}
  The product of the diagonal entries of
                    $B=B(\sigma_1)+B(\sigma_2)$ is a
                    polynomial of degree 12. Whereas a term in the
                    expansion of the determinant of $B$ with
                    off-diagonal entries has at most degree 10.
\end{example}

\begin{proof}[Proof of Theorem~\ref{theorem: invertibility of A}]

The score equation matrix $A$ associated to a maximal regular triangulation $\Delta$ can be written as 
  \begin{equation*}
     A= \sum_{\sigma \in \Delta } A(\sigma),
 \end{equation*}
where the entries of $A(\sigma)$ for $i\neq j$ are

 \begin{align*}
   A(\sigma)_{i,j}&= \vol(\sigma)\left(\prod_{\alpha\in \sigma \backslash \{j\}}\frac{1}{(y_j-y_\alpha)}\right)\left(\frac{1}{y_j-y_i}\right),\\
   A(\sigma)_{j,j} &=\vol(\sigma) \left(\prod_{\alpha\in \sigma \backslash  \{j\}}\frac{1}{(y_j-y_\alpha)}\right)\left(1-\sum_{\alpha\in \sigma \backslash \{j\}
   }\frac{1}{(y_j-y_\alpha)}\right).
 \end{align*}
 The matrix $A(\sigma)$ is sparse: If $i$ or $j$ does not belong to $\sigma$ then $A_{i,j}(\sigma)=0$.

   Let $B$ (resp. $B(\sigma)$) be the matrix that is obtained by multiplying the $j$-th column of $A$ (resp. $A(\sigma)$) by $\left(\prod_{\alpha \in \sN(j)}(y_j-y_{\alpha})^2\right)$ for $j=1,\ldots,n$:
  \begin{equation}
  \label{eq:Columns}
      B_{.\,,\,j}= A_{.\,,\,j}\left(\prod_{\alpha \in \sN(j)}(y_j-y_{\alpha})^2\right)= \sum_{\sigma \in \Delta } A(\sigma)_{.\,,\,j}\left(\prod_{\alpha \in \sN(j)} (y_j-y_{\alpha})^2\right).
    \end{equation}

 Fix $\sigma\in \Delta$. We describe separately the off-diagonal and diagonal entries of $B(\sigma)$. For $i,j\in \sigma$ and $i \neq j$ we get \begin{align*}
     B(\sigma)_{i,j}&= A(\sigma)_{i,j}\left(\prod_{\alpha\in \sigma \backslash \{j\}}(y_j-y_\alpha)^2\right)\left(\prod_{\alpha \in \sN(j)\backslash\sigma }(y_j-y_{\alpha})^2\right)\\
     &=\frac{\vol(\sigma)}{y_j-y_i}\left(\prod_{\alpha\in \sigma \backslash \{j\}}\frac{1}{(y_j-y_\alpha)} \prod_{\alpha\in \sigma \backslash \{j\}}(y_j-y_\alpha)^2\right) \left(\prod_{\alpha \in \sN(j)\backslash\sigma }(y_j-y_{\alpha})^2\right)\\
     &= \vol(\sigma)\left(\prod_{\alpha\in \sigma \backslash \{i,j\}}(y_j-y_\alpha)\right) \left(\prod_{\alpha \in \sN(j)\backslash\sigma }(y_j-y_{\alpha})^2\right).
 \end{align*}
 And for the diagonal entries
 \begin{align*}
     B(\sigma)_{j,j}&=A(\sigma)_{j,j} \left(\prod_{\alpha \in \sN(j) }(y_j-y_{\alpha})^2\right)\\
     &=\vol(\sigma)\left(\prod_{\alpha\in \sigma \backslash  \{j\}}\frac{1}{(y_j-y_\alpha)}\right)\left(1-\sum_{\alpha\in \sigma \backslash \{j\}}\frac{1}{(y_j-y_\alpha)}\right) \left(\prod_{\alpha \in \sN(j) }(y_j-y_{\alpha})^2\right)\\
     &=\vol(\sigma)\left(\prod_{\alpha\in \sigma \backslash  \{j\}}(y_j-y_\alpha)-\sum_{k\in \sigma \backslash \{j\}}\;\prod_{\alpha\in \sigma \backslash \{j,k\}}(y_j-y_\alpha)\right) \left(\prod_{\alpha \in \sN(j)\backslash\sigma }(y_j-y_{\alpha})^2\right). 
 \end{align*}
Given a polynomial $f\in \R[y_1,\ldots,y_n]$, we can rewrite $f=\sum_{i=0}^{d_j}f_i y_j^i$ as a univariate polynomial in $y_j$ of degree $d_j$, where $f_i\in \R[y_i: i\neq j]$ is a constant with respect to $y_j$. We then define the \emph{initial form} of $f$ with respect to $j$ to be  
$$ \text{in}_j(f)= f_{d_j}y_j^{d_j}.
$$
We observe that for the off-diagonal entries $B(\sigma)_{i,j}$, the initial form with respect to $j$ is 
 \begin{equation*}
     \text{in}_{j}(B(\sigma)_{i,j})= y_j^{2\gamma_j - d-1},
 \end{equation*}
 where $\gamma_j=\vert\sN(j)\vert$ is the number of vertices adjacent to $j$ in $\Delta$. 
 Whereas for the diagonal entry $B(\sigma)_{j,j}$, the initial form is
 \begin{equation*}
     \text{in}_{j}(B(\sigma)_{j,j})= y_j^{2\gamma_j- d}.
 \end{equation*}
 In both cases, the degree of the initial form is the degree of the polynomial. We sum the matrices $B(\sigma)$ for $\sigma\in \Delta$, to get $B$ and note that the coefficient of the monomial $y_j^{2\gamma_j -d}$ in $B_{j,j}$ is the number of simplices in $\Delta$ containing vertex $j$. Hence, using the Leibniz formula to compute the determinant of $B$, we get that the product of diagonal entries is a polynomial of degree
 $\displaystyle\left( \sum_{j=1}^n 2\gamma_j -d \right)$. All off-diagonal entries in that column of $B$ are of degree one smaller, thus any monomial in the expanded form of the
 determinant with off-diagonal entries must have degree at least two smaller than the product of diagonal entries. The following equality is a direct consequence of~(\ref{eq:Columns})
 \begin{equation*}\label{Eq: det C vs det A}
       \det\left(B\right)= \det\left(A\right)\prod_{j=1}^n \left(\prod_{\alpha \in \sN(j)}(y_j-y_{\alpha})^2\right). 
   \end{equation*}
 Since $\det(B)$ is not identically $0$, $\det(A)$ is not identically zero, hence $A$ is invertible over the field of rational functions. 
\end{proof}

The proof of Theorem \ref{theorem: invertibility of A} inspires the following  conjecture about the combinatorial properties of the determinant.

\begin{conjecture}
The sum over terms of highest total degree of the numerator of $\det(A)$ is
$$
\prod_{j=1,\ldots,n} \left(\sum_{\sigma \in \Delta \text{ s.t. } j \in \sigma} \vol(\sigma) \prod_{\alpha \in \sN(j) : \, \alpha \not \in \sigma} \left(y_j - y_{\alpha}\right)\right).
$$
\end{conjecture}

Since $A$ is invertible,~(\ref{equation:score_equations_matrix_form}) can be rewritten as
$$
e^y = A^{-1} w
$$
where entries of $A$ are rational functions in $\mathbb{R}(y_1,\ldots,y_n)$.

\begin{cor} \label{conjecture:polynomial-exponential-form}
Fix a maximal triangulation $\Delta$. Then the critical equations~(\ref{critical-equations}) can be written in the form
\begin{equation} \label{eqn:polynomial-exponential-form}
\begin{gathered}
    \exp(y_1)=p_1(y_1,y_2,\dots,y_n) \\
    \exp(y_2)=p_2(y_1,y_2,\dots,y_n) \\
    \vdots \\
    \exp(y_n) = p_n(y_1,y_2,\dots,y_n) 
\end{gathered}
\end{equation}
where $p_1,\ldots,p_n \in \mathbb{R}(y_1,\ldots,y_n)$. If $x_1,\ldots,x_n \in \mathbb{Q}^d$, then $p_1,\ldots,p_n \in \mathbb{Q}(y_1,\ldots,y_n)$.
\end{cor}

We will explore rational-exponential systems of the form~(\ref{eqn:polynomial-exponential-form}) further in Sections~\ref{section:one-dimension-one-cell}-\ref{section:three-points-on-a-line}. The following is a result from transcendental number theory, for a textbook reference see Theorem 1.4 of \cite{BakerTranscendentalNumberTheoryTEXT1990}.

\begin{theorem}[Lindemann-Weierstrass]\label{theorem:lindemann-weierstrass}
If $y_1,\dots,y_r$ are distinct algebraic numbers then the numbers $\exp(y_1), \dots, \exp(y_r)$ are linearly independent over the algebraic numbers.
\end{theorem}

A special case of the Lindemann-Weierstrass theorem is the Lindemann theorem which states that $\exp(y)$ is transcendental for algebraic $y \neq 0$.

\begin{theorem}\label{th:transcend}
Let $\mathcal{X}\subseteq \mathbb{Q}^d$.
If $\operatorname{vol}(\conv(\mathcal{X}) )\neq 1$, then there exists an open ball of weights $\mathcal{U}\subseteq \mathbb{R}^n$ such that for every $w \in \mathcal{U}$, 
at least one coordinate  of the optimal height vector $y^*$ is transcendental.
If $\operatorname{vol}(\conv(\mathcal{X}) )= 1$, then all coordinates of $y^*$ are algebraic if and only if $w$ is in the cone over the secondary polytope $\Sigma(X)$. 
\end{theorem}

\begin{proof}
Let $\Delta$ be a maximal regular triangulation. According to \cite[Theorem 1.2]{robeva2019geometry}, there exists an open ball $\mathcal{U}\subseteq \mathbb{R}^n$ of weights  that induces the maximal regular triangulation $\Delta$. Take any $w \in \mathcal{U}$ and consider the rational-exponential system~(3.2) for this choice of $\Delta$ and $w$. Then we have $\exp(y_1)=p_1(y_1,\ldots,y_n)$ where $p_1$ is a rational function in $\mathbb{Q}(y_1,\ldots,y_n)$. Assume that $y_1,\ldots,y_n$ are algebraic. By Lindemann's theorem $\exp(y_1)$ is algebraic if and only if $y_1=0$. 

However, $p(y_1,\ldots,y_n)$ is always algebraic, since $y_1,\ldots,y_n$ are algebraic and the algebraic numbers form a field. 
Hence $y_1=0$. We can argue similarly that $y_i=0$ for all $i$. The vector $y=(0,\ldots,0)$ belongs to the boundary of the Samworth body if and only if the volume of the convex hull of $X$ is $1$. In this case, $y$ is the optimal solution if $w$ is in the cone over the secondary polytope $\Sigma(X)$ by~\cite[Corollary 3.9]{robeva2019geometry}.
\end{proof}

\subsection{One cell in one dimension}\label{section:one-dimension-one-cell}

In this section we apply the invertibility of the score equation matrix to give a closed form solution to log-concave maximum likelihood estimator in case the logarithm of the optimal density is a linear function on the real line. If $X=(x_1,x_2) \subset \mathbb{R}$, then 
$$
A = \vol({\sigma}) \begin{bmatrix}
 \frac{1}{y_1-y_2} - \frac{1}{(y_1-y_2)^2}  &  \frac{1}{(y_1-y_2)^2}\\
  \frac{1}{(y_1-y_2)^2} & -\frac{1}{y_1-y_2} - \frac{1}{(y_1-y_2)^2} 
\end{bmatrix}
$$
and
$$
A^{-1} = \frac{1}{\vol(\sigma)}
\begin{bmatrix}
1 + y_1-y_2 & 1\\
1 & 1-y_1+y_2
\end{bmatrix}.
$$
Hence the polynomial-exponential system~(\ref{eqn:polynomial-exponential-form}) has the form
\begin{align}
\exp(y_1) &= \frac{1}{\vol(\sigma)} \left( \left(1+y_1-y_2\right)w_1 + w_2 \right) \label{eqn:1-dim-eqn1}\\
\exp(y_2) &= \frac{1}{\vol(\sigma)} \left( w_1 + \left(1-y_1+y_2\right)w_2 \right) \label{eqn:1-dim-eqn2}
\end{align}
Dividing~(\ref{eqn:1-dim-eqn1}) by~(\ref{eqn:1-dim-eqn2}) and setting $y_{12} = y_1 - y_2$, gives
\begin{equation} \label{eqn:1-dim-eqn12}
\exp(y_{12}) = \frac{ \left(1+y_{12}\right)w_1 + w_2 }{w_1 + \left(1-y_{12}\right)w_2}.
\end{equation}
In the rest of the section we will discuss how to solve Equation~(\ref{eqn:1-dim-eqn12}) using Lambert functions. The solutions for $y_1$ and $y_2$ can then be obtained from Equations~(\ref{eqn:1-dim-eqn1}) and~(\ref{eqn:1-dim-eqn2}) by solving for $y_{12}$. 

\begin{definition}[Section 2 in \cite{BarikzGeneralizations}]
For $x,t_i,s_j\in \mathbb{R}$, consider the function $$\exp(x)\frac{(x-t_1)(x-t_2)\ldots (x-t_n)}{(x-s_1)(x-s_2)\ldots (x-s_m)}.$$
We denote its (generally multi-valued) inverse function at the point $a\in\mathbb{R}$ by $$W(t_1,t_2,\ldots,t_n;s_1,s_2,\ldots,s_m;a)$$ and call it \emph{the generalized W-Lambert function}. The function $W(a):=W(0;;a)$ is called the usual \emph{W-Lambert function}.
\label{def: generalized Lambert}
\end{definition}

We have $W(;;a)=\log (a)$.

\begin{prop}\label{lemma:one_dim_one_cell}
The tent poles corresponding to a single-cell triangulation in $1$ dimension are given by:
\begin{align*} 
\begin{split}
y_1=\log(w_1 W(\rho+1;-\rho^{-1}-1;-\rho)+w_1+w_2) - \log(\vol(\sigma)), \\
y_2=\log(-w_2 W(\rho+1;-\rho^{-1}-1;-\rho)+w_1+w_2) - \log(\vol(\sigma)),
\end{split}
\end{align*}
where $\rho={w_1}/{w_2}$ and $W(\rho+1;-\rho^{-1}-1;-\rho)$ is a value of the  multi-valued generalized Lambert $W$ function if $y_1 \neq y_2$. Otherwise $y=(- \log(\vol(\sigma)),- \log(\vol(\sigma)))$.
\end{prop}

\begin{proof}
Recall from Equation~(\ref{eqn:1-dim-eqn12}):
\begin{equation*}
    \exp(y_{12})=\frac{w_1 y_{12} +w_1 +w_2}{-w_2 y_{12} +w_1 +w_2}
\end{equation*}
\noindent or, by setting $\rho = w_1/w_2$, equivalently 
\begin{equation*}
\frac{y_{12}-\rho-1}{y_{12} +\rho^{-1} + 1}\exp(y_{12}) = -\rho.
\end{equation*}

\noindent Seen as an equation in $y_{12}$ this has solutions given by the generalized Lambert function $W(\rho+1;-\rho^{-1}-1;-\rho)$. The solutions for $y_1$ and $y_2$ can then be obtained from~(\ref{eqn:1-dim-eqn1}) and~(\ref{eqn:1-dim-eqn2}) by solving $y_{12}$.
\end{proof}

\begin{remark}
Proposition~\ref{lemma:one_dim_one_cell} generalizes to the case when we have $n$ points on a line and the optimal tent function is supported on one cell.
\end{remark}

\begin{figure}
\centering
\includegraphics[scale=0.5]{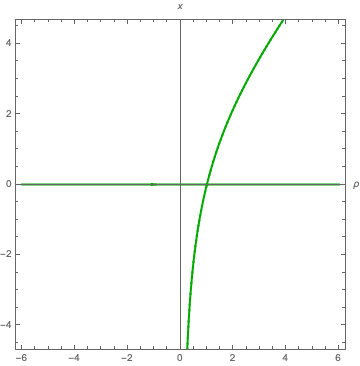}
\caption{Generalized Lambert function $W(\rho+1;-\rho^{-1}-1;-\rho)$.}
  \label{fig:implicitplotMaple}
\end{figure}

The generally multi-valued generalized $W$-Lambert function $W(\rho+1;-\rho^{-1}-1;-\rho)$ is plotted in Figure~\ref{fig:implicitplotMaple}.
We explore its branches, i.e.,~single-valued functions of $\rho$, using $r$-Lambert functions. 

\begin{definition}[Section 3.2 in~\cite{BarikzGeneralizations}]
If $r\in\mathbb{R},$ consider the function $$x\exp(x) +rx.$$
We denote its inverse function in the point $a\in\mathbb{R}$ by $W_r(a)$ and call it the $r$\emph{-Lambert function}.
\end{definition}

The following theorem makes the connection between the generalized Lambert function and the $r$-Lambert function:
\begin{theorem}[Theorem 3 in \cite{BarikzGeneralizations}]\label{Th:rLambert}
If $t,s,a\in\mathbb{R}$, the following equality holds:
$$W(t;s;a)=t+W_{-a\exp(-t)}\big (a\exp(-t)(t-s)\big ).$$
\end{theorem}
Hence
$$W(\rho+1;-\rho^{-1}-1;-\rho)=\rho+1+W_{\rho\exp(-\rho-1)}\big (-\rho\exp(-\rho-1)(\rho+\rho^{-1}+2)\big ).$$
The number of branches of the $r$-Lambert function is classified in~\cite[Theorem 4]{BarikzGeneralizations} and~\cite[Theorem 4]{Maignan}. For $r=\rho\exp(-\rho-1)$, it translates to
\begin{enumerate}
    \item two branches, if $\rho\exp(-\rho-1)< 0$; 
    \item three branches, if $0<\rho\exp(-\rho-1)<\exp(-2)$;
    \item one branch, if $\rho\exp(-\rho-1)\ge \exp(-2)$.
\end{enumerate}
The second case happens when $\rho > 0$, in which case we have the double branch of constant zero function and an additional branch. 
This is the branch that is relevant to us in the context of Proposition~\ref{lemma:one_dim_one_cell}.  The first case happens when $\rho < 0$, in which case there exists a double branch of the constant zero function. 
This cannot appear for positive weights $w_i$.
The third case does not happen.

The $r$-Lambert function can be computed with the \texttt{C++} implementation~\cite{Mezo2017}. Alternatively, one can use results about computing roots of polynomial-exponential equations.
In \cite{Maignan:1998}, a symbolic-numeric algorithm is proposed for constructing explicitly an interval containing all the real roots of a single real polynomial-exponential equation, and counting how many roots are contained in a non-bounded interval. 
In \cite{MR703463}, the decision problem of the existence of positive roots of such functions is discussed. This subject is strongly related to quantifier elimination~\cite{MR3283842}, and to transcendentality problems~\cite{MR2854845,MR3776771,MR3577161}. The latter problem of the transcendence theory appears in our Theorem \ref{th:transcend}. 

\subsection{Two cells in one dimension} \label{section:three-points-on-a-line}

Let $X=(x_1,x_2,x_3) \subset \mathbb{R}$. Then
$$
A = \begin{bmatrix}
\frac{v_1}{(y_1-y_2)^2}-\frac{v_1}{y_1-y_2} & -\frac{v_1}{(y_1-y_2)^2} & 0\\
-\frac{v_1}{(y_1-y_2)^2} & \frac{v_1}{(y_1-y_2)^2} -\frac{v_1}{y_1-y_2} + \frac{v_2}{(y_2-y_3)^2} - \frac{v_2}{y_2-y_3} & - \frac{v_2}{(y_2-y_3)^2}\\
0 & - \frac{v_2}{(y_2-y_3)^2} & \frac{v_2}{(y_2-y_3)^2} - \frac{v_2}{y_2-y_3}
\end{bmatrix}.
$$
Recall $y_{12}=y_1-y_2$ and $y_{23}=y_2-y_3$. Then 
\begin{scriptsize}
$$
A^{-1}= \frac{1}{v_1(1+y_{23}) + v_2(1-y_{12})} \begin{bmatrix}
-(1+y_{12})(1+y_{23}) + \frac{v_2}{v_1}y_{12}^2& -1-y_{23} & -1\\
-1-y_{23} & (-1+y_{12})(1+y_{23}) & -1+y_{12}\\
-1 &  -1+y_{12} & -(-1+y_{12})(-1+y_{23})+\frac{v_1}{v_2} y_{23}^2
\end{bmatrix}.
$$
\end{scriptsize}

Consider the polynomial-exponential system $\exp(y)=A^{-1} w$ as in~(\ref{eqn:polynomial-exponential-form}). Dividing the first equality with the second one and the second one with the third one gives:
\begin{equation}\label{eq:2eq}
\begin{dcases}
    \exp(y_{12})  = \frac{(-(1+y_{12})(1+y_{23}) + \frac{v_2}{v_1}y_{12}^2) w_1 +(-1-y_{23}) w_2 - w_3}{(-1-y_{23}) w_1 +(-1+y_{12})(1+y_{23}) w_2 +(-1+y_{12}) w_3},\\[2ex]
     \exp(y_{23})  = \frac{(-1-y_{23}) w_1 +(-1+y_{12})(1+y_{23}) w_2 +(-1+y_{12}) w_3}{-w_1 +(y_{12}-1) w_2 - ((y_{12}-1)(y_{23}-1)+\frac{v_1}{v_2} y_{23}^2) w_3}.
\end{dcases}
\end{equation}

Hence we could reduce a polynomial-exponential system with three equations and three variables to a polynomial-exponential system with two equations and two variables. Systems of two rational bivariate polynomial-exponential equations such as~(\ref{eq:2eq}) are studied in~\cite{Maignan:1998}. An algorithm giving the number of solutions of such a system is provided, where all the solutions are contained in a generalized open rectangle of type $I_1\times I_2\subset \mathbb{R}^2$, under the hypothesis that at least one of the intervals $I_1$ or $I_2$ is bounded.

\begin{remark}
Let $X \subset \mathbb{R}$. If we consider tent functions $h_{X,y}$ that are supported on two cells such that $h_{X,y}$ is a constant function on one of the two cells, then one can use methods similar to the one cell case (see Section~\ref{section:one-dimension-one-cell}) to give the optimal solution using the Lambert function.
\end{remark}

\section{Certifying solutions with Smale's \texorpdfstring{$\alpha$}{alpha}-theory}\label{section:Smale-alpha-certify}

As explained in Section~\ref{sect:LogConcMLE}, our task is to maximize the objective function $S(y_1,\dots,y_n)$ defined in Corollary \ref{cor:optsecondary}. For a subdivision~$\Delta$, we can find the optimal $y^*_{\Delta}$ by considering $S_{\Delta'}(y_1,\dots,y_n)$ for any maximal triangulation $\Delta'$ that refines $\Delta$, substituting $y_i$ that can be expressed in terms of other $y$'s for the subdivision $\Delta$ and solving the system of critical equations $\partial \widetilde{S}_{\Delta}/\partial y_i = 0$ for the resulting function $\widetilde{S}_{\Delta}$. For maximal triangulations, we have $\widetilde{S}_{\Delta}=S_{\Delta}$ and the system of critical equations is given by~(\ref{critical-equations}). We will write $S_{\Delta}$ instead of $\widetilde{S}_{\Delta}$ also when talking about general subdivisions and for brevity we denote the system of critical equations by $\nabla S_{\Delta}(y) = 0$.   We say the system is square because we have $n$ equations $\partial S_{\Delta}/\partial y_i = 0$ in $n$ variables $y_1,\dots,y_n$. Usually it will be impossible to write down exact solutions to these systems, but there is a way forward. In what follows we discuss the computation of certified solutions to this system of equations. To do so, we discuss Smale's $\alpha$-theory, which makes mathematically rigorous the idea of approximate zeros in the sense of quadratic convergence of Newton iterations. The following influential definition was given in \cite{BCSS1998, S1986}.

\begin{definition}[Chapter 8 of \cite{BCSS1998}]\label{def:approximate-zero}
Let $Df(x)$ be the $n \times n$ Jacobian matrix of the square system of complex-analytic equations $f(x)=0 \in \mathbb{C}^n$, where $f:\mathbb{C}^n \to \mathbb{C}^n$ is written as a column vector of its component functions
$$f(x) = [f_1(x_1,\dots,x_n), \dots, f_n(x_1,\dots,x_n)]^T .$$
A point $z \in \mathbb{C}^n$ 
is an {\em approximate zero} of $f\,$ if there exists a zero $z^* \in \mathbb{C}^n$ of~$f$ such that the sequence of  Newton iterates
$$z_{k+1} = z_k - Df(z_k)^{-1} f(z_k)
$$
satisfies
$$ \Vert z_{k+1} - z^* \Vert \,\leq \, \frac{1}{2}\Vert z_k- z^* \Vert^2$$
for all $k \geq 1$ where $z_0 = z$. If this holds, then we call $z^*$ the {\em associated zero} of~$z$.
Here~${\Vert x\Vert := (\sum_{i=1}^n x_i\overline{x_i})^\frac{1}{2}}$
is the standard norm in~$\mathbb{C}^n$, and the zero $z^*$ is assumed to be
nonsingular, meaning that ${\rm det}Df( z^* ) \neq 0 $.
\end{definition}

Therefore the problem becomes two-fold. Given a system of equations $f$, we need a way to (1) generate approximate solutions, and (2) certify their quadratic convergence under Newton iterations. The methods of Smale's $\alpha$-theory solve exactly this second problem. 
This is accomplished using the constants $\alpha(f,x), \beta(f,x)$ and $\gamma(f,x)$, which we will discuss in Section~\ref{sec:smale-alpha-theory}. Typically $\gamma$ is difficult to compute, since it is defined as the supremum of infinitely many quantities depending on higher-order derivatives of our system of equations. However, explicit upper bounds on $\gamma$ were calculated in \cite{HL2017} which we can specialize to the system required for log-concave density estimation. These upper bounds have the advantage that they are easily computed from our system $\nabla S = 0$, and can therefore be used to $\alpha$-certify approximate solutions coming from numerical software. In Section \ref{sec:smale-alpha-theory}, we make this precise, discussing recent work on the subject~\cite{HL2017, HS2012, SS1993, S1986} and how it applies in our context.

\begin{remark} \label{example:evaluating-numerical-solutions}
One might wonder why we do not directly evaluate the equations in question
to the approximate height values given by statistical packages. 
The reason is that we want to have a measure of how accurate this solution is,
which is also very sensitive to the system.
Consider for example the system consisting of the single polynomial $f(x) = x$. 
We would not accept $1/2$ as a solution.
But if we consider the system $f(x) = x^{10}$ and we evaluate at $x=1/2$,
we get a value that is less than $0.001$.
This could have been tempting, but note that in both cases the difference between actual solution and approximation is the same.

Another example that illustrates the
potential difficulties involved in judging a numerical solution based on its evaluation into the original system of equations comes from~\cite{BatesHauensteinSommeseWampler2013BertiniTEXT}. Consider the univariate polynomial
\begin{equation*}
    f(z) = z^{10} - 30z^9 + 2.
\end{equation*}
A solution which is accurate within $9.4 \times 10^{-12}$ of the true solution is
\begin{equation*}
    z^* = 30.00000000000142 - 0.00000000000047i,
\end{equation*}
but evaluating the polynomial at this solution yields a complex number $f(z^*)$ with norm $|f(z^*)| = 31.371$, which certainly seems far from zero. However, refining the accuracy of this solution to
\begin{equation*}
    z^{**} = 29.9999999999998983894731343124 + 0.0000000000000000000000062i,
\end{equation*}
we find that $|f(z^{**})| = 0.00000000032$, which is much better.
\end{remark}

\subsection{Smale's \texorpdfstring{$\alpha$}{alpha}-theory}\label{sec:smale-alpha-theory}

The intuition behind $\alpha$-theory is as follows. The size of the initial Newton iteration step combined with the size of the derivatives control how quickly Newton iteration converges to a true solution. We can calculate the size of the Newton iteration step, so if we have some control over the higher order derivatives of $f$, then we should be able to certify whether a solution satisfies the criterion of Definition \ref{def:approximate-zero}. This motivates the definition of the following constants $\alpha, \beta, \gamma \in \mathbb{R}$, associated to a system of equations $f$ at a point $x$. These constants measure quantities relevant to certifying approximate zeros.

\begin{definition}\label{def:alpha-beta-gamma}
Let $f:\mathbb{C}^n \to \mathbb{C}^n$ be a system of complex-analytic functions and let $x \in \mathbb{C}^n$. We define $\alpha(f,x)$ to be the product of $\beta(f,x)$ and $\gamma(f,x)$:
\begin{equation*}
    \alpha(f,x) = \beta(f,x) \gamma(f,x).
\end{equation*}
The constant $\beta(f,x)$ measures the size of the Newton iteration step applied at $x$, namely:
\begin{equation*}
    \beta(f,x) = \| Df(x)^{-1} f(x) \|,
\end{equation*}
while $\gamma(f,x)$ bounds the sizes of the following quantities, involving the higher order derivatives:
\begin{equation*}
    \gamma(f,x) = \text{sup}_{k\geq2} \left\| \frac{Df(x)^{-1} D^k f(x)}{k!} \right\|^{\frac{1}{k-1}}.
\end{equation*}
\end{definition}

If we can compute these constants $\beta, \gamma$ for a candidate solution, then we can utilize the following

\begin{theorem}[Chapter 8 of \cite{BCSS1998}]\label{theorem:alpha-theorem}
If $f:\mathbb{C}^n \to \mathbb{C}^n$ is a system of complex-analytic functions and $x \in \mathbb{C}^n$ satisfies
\begin{equation*}
    \alpha(f,x) < \frac{13-3\sqrt{17}}{4} \approx 0.157671,
\end{equation*}
then $x$ is an approximate zero of $f = 0$.
\end{theorem}

For polynomial systems, all higher-order derivatives eventually vanish. Exactly this fact was used in \cite{SS1993} to derive an upper bound for $\gamma(f,x)$ which involves the degrees of the polynomials in the system $f$. This is highly convenient since, even for systems of polynomials, calculating $\gamma(f,x)$ purely based on the definition is quite a difficult task. Yet, if we are to certify candidate solutions to our system of equations, we need to calculate $\gamma$ and $\beta$ at our candidate $x$, multiply them, and hope they are below $\approx 0.157671$.

\subsection{Polynomial-exponential systems}\label{subs:extend}

For polynomial-exponential systems $f$, calculating $\gamma(f,x)$ is even harder. However, in \cite{HL2017}, an upper bound was computed for $\gamma$ involving quantities more readily apparent in a given system $f$ than what appears in the bare definition of $\gamma$. In fact, an upper bound for $\gamma$ is calculated which applies to a general class of systems, as well as upper bounds for several special cases. One of these special cases can be further specialized to the system of equations $\nabla S = 0$ arising in log-concave density estimation (this is Lemma \ref{lemma:reformat-critical-equations} below). In \cite{HL2017} an example is given where the bounds for the special cases allowed candidate solutions to be $\alpha$-certified despite failure using the more general bounds. In this section we summarize the results of \cite{HL2017} as they relate to log-concave density estimation. First we need a few definitions.
\begin{definition}\label{def:one-norm-on-Cn}
For a point $x \in \mathbb{C}^n$ define
\begin{equation*}
    \|x\|^2_1 = 1 + \|x\|^2 = 1 + \sum_{i=1}^n |x_i|^2.
\end{equation*}
For a polynomial $g:\mathbb{C}^n \to \mathbb{C}$ given as $g(x) = \sum_{|\rho| \leq d} a_{\rho} x^{\rho}$ define
\[
\|g\|^2 = \frac{1}{d!} \sum_{|\rho|\leq d} \rho!\cdot(d - |\rho|)!\cdot|a_\rho|^2.
\]
For a polynomial system $f: \mathbb{C}^n \to \mathbb{C}^n$
with $f(x) = [f_1(x),\dots,f_n(x)]^T$, we define
\[
\|f\|^2 = \sum_{i=1}^n \|f_i\|^2.
\]
\end{definition}

We now define a quantity $\mu(f,x)$ associated to a polynomial system which will play a role in bounding $\gamma$ later.
\begin{definition}\label{def:mu-for-polynomial-system}
Let $f:\mathbb{C}^n \to \mathbb{C}^n$ be a polynomial system with $\text{deg } f_i = d_i$. Define
\begin{equation*}
    \mu(f,x) = \text{max } \left\{ 1, \|f\| \cdot \|Df(x)^{-1} C_{f}(x) \| \right\}
\end{equation*}
where $C_f(x)$ is the diagonal matrix
\begin{equation*}
C_{f}(x) = \left[\begin{array}{ccc} d_1^{1/2}\cdot\|x\|_1^{d_1-1} & & \\ & \ddots & \\ & & d_n^{1/2}\cdot\|x\|_1^{d_n-1} \end{array}\right].
\end{equation*}
\end{definition}
Following \cite{HL2017}, we extend Definition \ref{def:mu-for-polynomial-system} to certain polynomial-exponential systems.
\begin{definition}\label{def:mu-for-poly-exponential-system}
Let $a \in \mathbb{Z}_{\geq0}$,
$\delta_i \in \mathbb{C}$,
and $\sigma_i \in \{1,\dots,n\}$. Consider the polynomial-exponential system
\begin{equation}
G(x_1,\dots,x_n,u_1,\dots,u_a ) = \left[\begin{array}{c}
P(x_1,\dots,x_n,u_1,\dots,u_a) \\
\begin{array}{cc}
u_1 - \exp(\delta_1 x_{\sigma_1}) \\
u_2 - \exp(\delta_2 x_{\sigma_2}) \\
\vdots \\ 
u_a - \exp(\delta_a x_{\sigma_a})
\end{array}
\end{array}\right], \label{Eq:ReductionG}
\end{equation}
where $P:\mathbb{C}^N \to \mathbb{C}^n$ is a polynomial system with $N = n+a$ variables. Thus, the system $G$ is a square system of size $N$.  We write $X:=(x,u)$. Define
\begin{equation*}
    \mu(G,X) = \max\left\{1, \left\|DG(x,u)^{-1} \left[\begin{array}{cc} C_{P}(x,u)\|P\| & \\ & I_a \end{array}\right]\right\|\right\}.
\end{equation*}
\end{definition}

The following specializes Corollary 2.6 of \cite{HL2017}. \begin{theorem}\label{theorem:specialized-gamma-bounds}
Let $a \in \mathbb{Z}_{\geq0}$,
$\delta_i \in \mathbb{C}$,
and $\sigma_i \in \{1,\dots,n\}$ and consider the polynomial-exponential system~(\ref{Eq:ReductionG}).
  Let $d_i = \text{deg } P_i$ and $D = \text{max } d_i$. For any $\lambda, \theta \in \mathbb{C}$ define
\begin{equation*}
    A(\lambda,\theta) = \text{max} \left\{ |\lambda|, \left| \frac{\lambda^2 \exp(\lambda \theta)}{2} \right| \right\}.
\end{equation*}
Then, for any $X = (x,u) \in \mathbb{C}^N$ such that the Jacobian of $G$ is invertible,
\begin{equation}\label{equation:gamma-bounded-by}
    \gamma(G,X) \leq \mu(G,X) \left( \frac{D^{3/2}}{2\|X\|_1} + \sum_{i=1}^a A(\delta_i, x_{\sigma_i}) \right).
\end{equation}
\end{theorem}
\begin{proof}
This is a straight-forward specialization of Corollary 2.6 of \cite{HL2017}. We set to zero quantities that deal with functions not relevant to log-concave density estimation.
\end{proof}

Therefore, reformulating our system of polynomial-exponential equations $\nabla S_{\Delta} = 0$ in the format~(\ref{Eq:ReductionG}) will allow us to calculate an upper bound on $\gamma$, which will allow us to certify solutions to our critical equations.

\begin{lemma}\label{lemma:reformat-critical-equations}
Fix a maximal regular triangulation $\Delta$. The polynomial-exponential system $ \nabla S_{\Delta} = 0 $
 can be reformulated as a system of equations of the form~(\ref{Eq:ReductionG}), demonstrating that Theorem \ref{theorem:specialized-gamma-bounds} applies in the context of log-concave maximum likelihood estimation.
\end{lemma}

\begin{proof}
The partial derivatives $\partial S_{\Delta}/\partial y_k$ are rational functions of the $y_i$ and the $\exp(y_i)$. Since we set each partial derivative to zero, we can clear denominators, creating a system of equations, each of which is a polynomial in the $y_i$ and the $\exp(y_i)$. Setting each $\delta_i = 1$ in~(\ref{Eq:ReductionG}), we can replace each occurrence of $\exp(y_i)$ with $u_i$, creating the polynomial system $P(y_1,\dots,y_n,u_1,\dots,u_n)$, hence $a = n$ as well. Appending the equations $u_i - \exp(y_i)$ to the system of polynomials $P$, we have a system of $2n$ equations in $2n$ unknowns. This system is of the required form in order to apply Theorem \ref{theorem:specialized-gamma-bounds}.
\end{proof}

Thus, we have everything we need to compute the upper bound in~(\ref{equation:gamma-bounded-by}) for a system of critical equations $\nabla S_{\Delta} = 0$ when $\Delta$ is a maximal regular triangulation. By calculating this upper bound for a given system of equations, we can certify approximate numerical solutions obtained in any way. 
When $\Delta$ is not a maximal regular triangulation, one must impose further linear constraints on some of the $y_i$,
as was the case in Example \ref{ex:2-5-7-critical-equations}.
After simplifications, 
one might still end up with terms involving exponentials of fractional convex combinations of the $y_i$.
This poses no threat for the purposes of $\alpha$-certification, 
as one may in fact use products of exponentials of the form $e^{\beta y_i}$.
In particular, a bound for $\gamma(G,X)$ also for these more general polynomial-exponential systems is given in~\cite[Corollary 2.6]{HL2017}.

In algebraic statistics, it is common to find algebraic invariants which characterize algebraic complexity. For example, the maximum likelihood degree of a statistical model gives information about the critical points of the likelihood function of a parametric model \cite{amendola2015maximum}.  Similarly, in nonparametric algebraic statistics, it could be the case that the combinatorial complexity of the optimal subdivision gives us information about the computational complexity of finding a numerical solution. 
\begin{question}
Does increasing the combinatorial complexity of the optimal subdivision decrease the likelihood that the numerical output from \texttt{LogConcDEAD} is $\alpha$-certified?
\end{question}

We study this question experimentally in the next section. In future work, one could hope to precisely describe this phenomenon, should it exist. Of course, higher degrees, more variables, more equations will always increase the bound on $\gamma$ we calculate, but the combinatorics should still play some role.

\subsection{A procedure for \texorpdfstring{$\alpha$}{alpha}-certifying}

One of our motivating questions was to determine the correct subdivision for a given data set, as was the case in Example \ref{motivatingexample}. In this section we describe a procedure based on Smale's $\alpha$-theory that in principle allows us to find the certifiably correct subdivision. Recall that the objective function $S(y_1,\dots,y_n)$ depends on a subdivision of the convex hull of the data set $X$. If there are $m$ subdivisions, then there are $m$ different objective functions $S_1,\dots, S_m$, and $m$ different possible systems of equations $\nabla S_1 = 0, \dots, \nabla S_m = 0$. Given an estimate of a solution $y^*$, perhaps computed numerically using existing software, we can attempt to $\alpha$-certify that solution using any of these systems as input to Lemma \ref{lemma:reformat-critical-equations} and Theorem \ref{theorem:specialized-gamma-bounds}. As we collect $\alpha$-certified critical points for the various objective functions, we can use this data to determine the correct subdivision, helping to answer our motivating question. 

In practice, we have found that numerically computed solutions $y^*$ are often not $\alpha$-certified, using any of the systems $\nabla S_i = 0$. However, using a brute-force search over all possible additional digits, we often can find one system $\nabla S_j = 0$ to which $y^* + \varepsilon$ is an $\alpha$-certified solution. Here, $\varepsilon = (\varepsilon_1,\dots, \varepsilon_n)$ is a vector providing additional digits of precision to each component of $y^*$. As we compute $\alpha$-values for each $y^* + \varepsilon$, we move in the direction which causes a decrease in the computed $\alpha$-value, until we are able to find an $\alpha$-certified $y^* + \varepsilon$. We describe this in the following

\vspace{3pt}
\begin{algorithm}[H]
    \DontPrintSemicolon
    \KwIn{A system $\nabla S_i = 0$ coming from the $i$th candidate subdivision and a candidate approximate solution $y^* = (y_1, \dots, y_n)$.}
    \KwResult{A refinement of the heights $y^* + \varepsilon$ along with alpha certification of the system, or inability to certify.}
    
    Let $p$ be the number of trusted significant digits (in binary) of the approximate solution $y^*$. \;
    Expressing $y^*$ in binary, compute the $\alpha$-value for all $3^n$ points $y_i + \epsilon_i 2^{-p}, \epsilon_i \in \{-1,0,1\}$. Keep the point with the lowest alpha value, and set this as the new  $y_i$.\;
    If the alpha value is below $0.157671$ stop and return the solution. If it has decreased between steps or remained the same, increase $p$ by $1$ and go to step $2$. If there is no improvement for several loops in a row, stop and declare inability to certify the system. \;
    \caption{Testing certifiability by digit refinement}
    \label{algorithm:refinement}
\end{algorithm}
\vspace{12pt}

\begin{remark}
Here we collect a few comments on Algorithm \ref{algorithm:refinement}.
\begin{enumerate}
    \item We note that this brute-force search over all possible digits could be replaced by any numerical procedure for finding solutions to a given set of equations, see for instance the \emph{refine} command in the Numerical Algebraic Geometry package for Macaulay2   \cite{MR2881262}. For example, Newton iteration could be used on the system of equations to produce more accurate solutions, which could then be $\alpha$-certified. However, to compare the outputs of \texttt{LogConcDEAD} for problems of increasing combinatorial complexity (see Table \ref{tab:bindig}), we wanted to use a completely ``blind'' brute-force search as described above.
    \item One does not need to stop at Step $3$ once a solution is certified. 
    Repeating the loop allows increasing the precision of the solution by moving to lower $\alpha$ values. This is in contrast to statistical software like \texttt{LogConcDEAD} which only allows up to $7$ significant digits.
    \item Although precision can be added, our (first) goal with Algorithm \ref{algorithm:refinement} is to find the correct subdivision induced by the heights. One can test several subdivisions here, therefore we say that we test the (approximate) solution against the corresponding system of equations. 
    \item It might happen that the $\alpha$-value does not immediately decrease from one loop to the next even if we have the correct system of equations. One reason is that if the next significant digit is a zero for all heights, we are computing an $\alpha$-value for the same point multiple times. 
    \item In step 1 of the above algorithm, we let $p$ be the number of trusted significant digits of the approximate solution $y^*$. We have found that several of the last digits of a solution computed with \texttt{LogConcDEAD} were incorrect, in the sense that if we start our search (in Algorithm \ref{algorithm:refinement}) earlier in the significant digits of $y^*$ we are able to $\alpha$-certify some $y^* + \varepsilon$. In this way, we can correct for some of the imprecision of a numerical solver.
\end{enumerate}
\end{remark}

\begin{example} \label{examplefor}
Consider the data set 
$X = (2,5,7)$ with weights $w=(\frac{1}{3},\frac{1}{2},\frac{1}{6})$. With this input, the package \texttt{LogConcDEAD} returns the heights 
\[
y^* = (y_1,y_2,y_3)= (-1.454152,-1.605833, -1.888083), 
\]
suggesting that there are two regions of linearity (Figure \ref{fig43a}). Let $\Delta=\{\{1,2\},\{2,3\}\}$. We consider critical equations for
$$S_{\Delta}(y_1,y_2,y_3)=\frac{\it y_1}{3}+\frac{\it y_2}{2}+\frac{\it y_3}{6}-3\,{\frac {{{\rm e}^{{\it y_1}}}-{
{\rm e}^{{\it y_2}}}}{{\it y_1}-{\it y_2}}}-2\,{\frac {{{\rm e}^{{\it y_2}
}}-{{\rm e}^{{\it y_3}}}}{{\it y_2}-{\it y_3}}}
$$
which lead to the polynomial-exponential system $\nabla S_{\Delta}: \mathbb{C}^3 \to \mathbb{C}^3$ given by
\begin{align*}
    (y_1-y_2)^2 \frac{\partial S(y_1,y_2,y_3)}{\partial y_1}&=0\\
    (y_1-y_2)^2 (y_2-y_3)^2 \frac{\partial S(y_1,y_2,y_3)}{\partial y_2}&=0\\
    (y_2-y_3)^2 \frac{\partial S(y_1,y_2,y_3)}{\partial y_3}&=0,
\end{align*}
where we have cleared denominators.
The numerical solution from \texttt{LogConcDEAD} is not immediately $\alpha$-certified, but after applying Algorithm \ref{algorithm:refinement} we obtain the $\alpha$-certified solution:
$y^* + \varepsilon = (y_1,y_2,y_3)=(-1.45415181, -1.60583278, -1.88808307).$
\end{example}

\begin{example} \label{exampleagainst}
We now consider the same sample $X = (2,5,7 )$ with uniform weights. As discussed in Example~\ref{ex:2-5-7-critical-equations}, \texttt{LogConcDEAD} output suggests that the logarithm of the optimal density has a single region of linearity (Figure \ref{fig43b}). Can we certify this assessment? Recall that substituting $y_2 = \frac{2}{5}y_1 + \frac{3}{5} y_3$ to $S(y_1,y_2,y_3) = \frac{1}{3} y_1 + \frac{1}{3} y_2 + \frac{1}{3} y_3 - 3 \frac{e^{y_1} - e^{y_2}}{y_1-y_2} -2 \frac{e^{y_2} - e^{y_3}}{y_2-y_3}$ gives
\begin{equation*}
    \widetilde{S} = \frac{7}{15} y_1 + \frac{8}{15} y_3 - 5 \frac{e^{y_1} - e^{y_3}}{y_1-y_3}.
\end{equation*}
The system of equations $\nabla \widetilde{S} = 0$ does have solutions, and we were able to check that the numerical solution $y^*$ computed by \texttt{LogConcDEAD} is an $\alpha$-certified solution to this amended system of equations. 
\end{example}

\begin{figure*}
  \begin{subfigure}[b]{0.3\textwidth}
  \includegraphics[trim={1cm 1cm 0 0},clip,width=1\textwidth,height=0.8\textwidth]{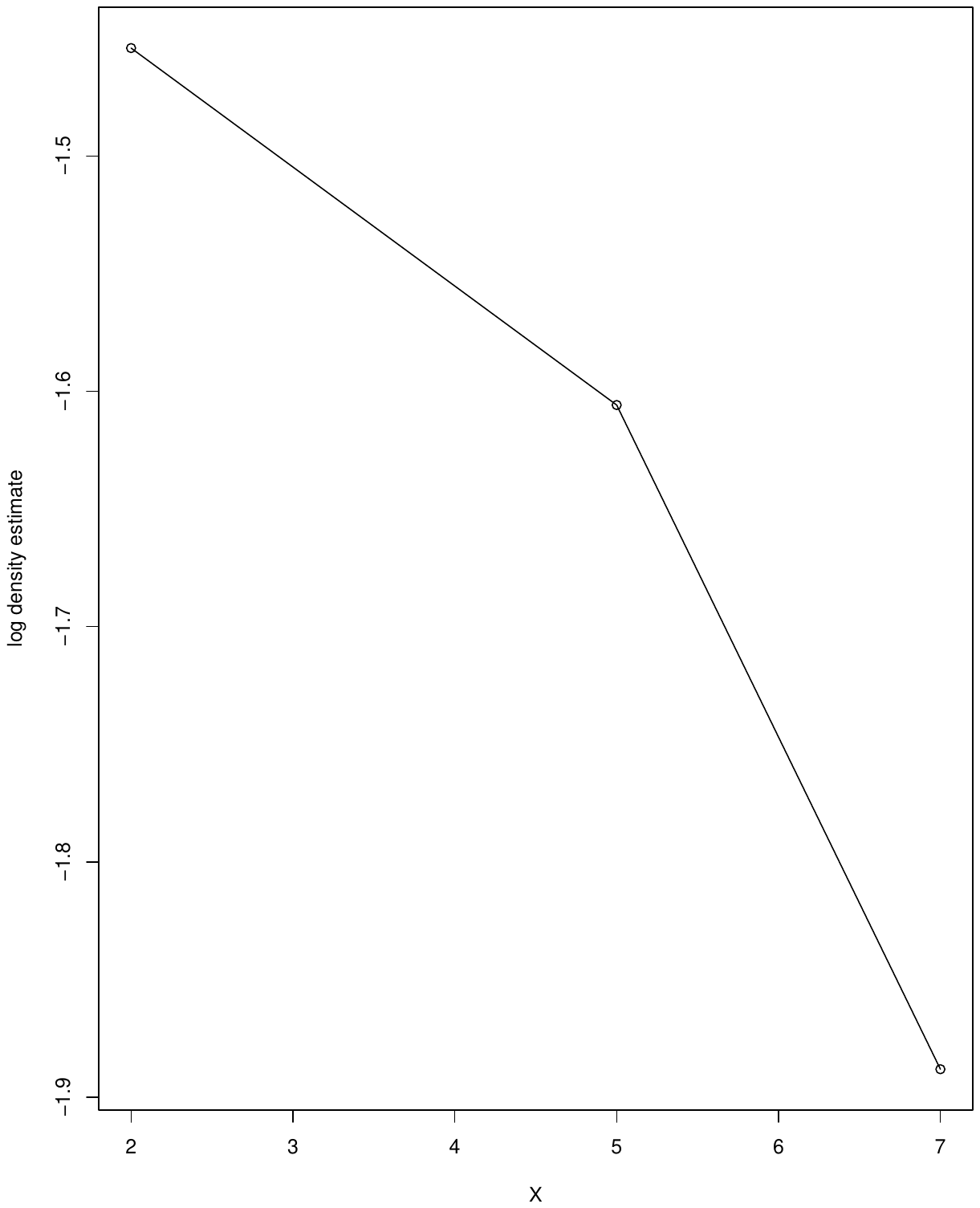}
    \caption{}
    \label{fig43a}
  \end{subfigure}
  \hfill
  \begin{subfigure}[b]{0.3\textwidth}
    \includegraphics[trim={1cm 1cm 0 0},clip,width=1\textwidth,height=0.8\textwidth]{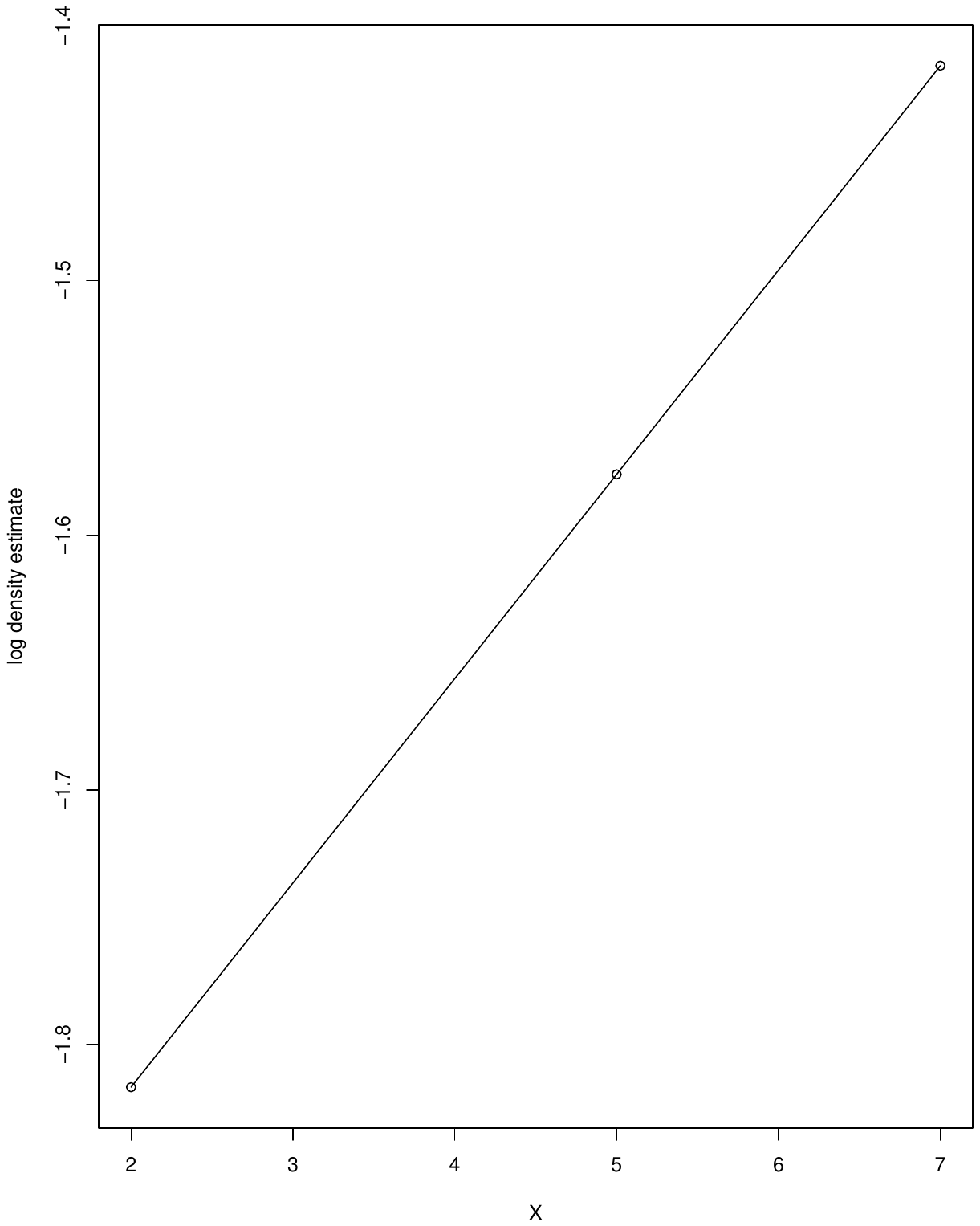}
    \caption{}
    \label{fig43c}
  \end{subfigure}     
  \hfill
  \begin{subfigure}[b]{0.3\textwidth}
    \includegraphics[width=1\textwidth,height=0.8\textwidth]{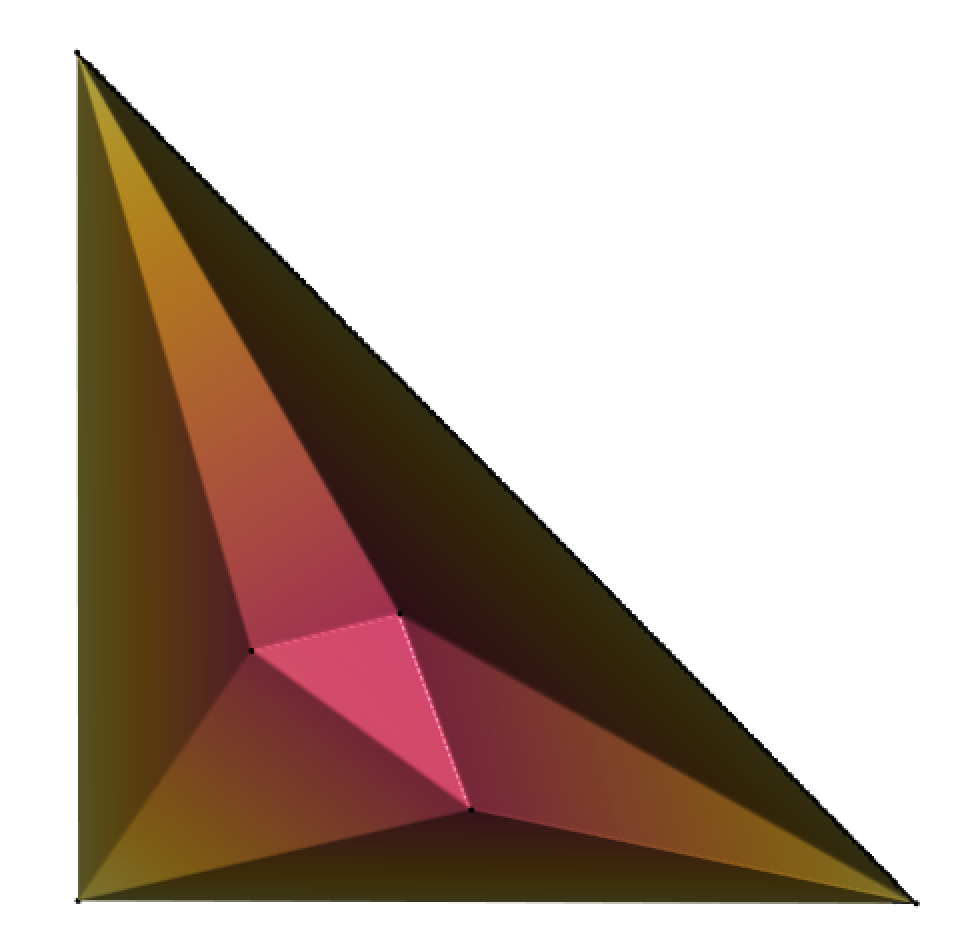}
    \caption{}
    \label{fig43b}
  \end{subfigure}
\caption{\label{heightgraphs}The height functions for (a) Example \ref{examplefor}; (b) Example \ref{exampleagainst}; (c) Example \ref{example2dim}}
\end{figure*}

\begin{example}
We used Algorithm \ref{algorithm:refinement} to certify the sample $X = (0,1,2, \dots, n ) \subset \mathbb{R}$ for weights given by the binomial distribution with $p=6/11$, i.e.,\ $w_i = \binom{n}{i} (6/11)^i(5/11)^{n-i}$. 
Looking at the \texttt{LogConcDEAD} output, we suspect that the triangulation given by the points consists of all consecutive line segments $\{i-1, i\}$ for $i \in 1, 2, \dots, n$. 
We therefore compute $\alpha$-values using the system of equations corresponding to the full triangulation. In all cases tested, we were able to certify the system for some refinement of the original \texttt{LogConcDEAD} output. In Table \ref{tab:bindig}, we summarize the number of binary digits required for certification in each case. This table suggests that the complexity of $\alpha$-certifying increases when the number of sample points increases.

\begin{table}[ht]
\centering
\begin{tabular}[t]{|c|c|c|c|c|c|}
\hline
n &3 & 4 & 5 & 6 & 7 \\
\hline
\text{binary digits} & 22 & 23 & 27 & 31 & 31\\
\hline
\end{tabular}
\caption{Number of binary digits needed to certify $n+1$ points with weights coming from an asymmetric binomial distribution.}
\label{tab:bindig}
\end{table}
\end{example}

We now present an example in two dimensions that needs more significant digits than the previous cases.

\begin{example} \label{example2dim}
We consider the point configuration from~\cite[Example 1.1]{robeva2019geometry}, 
given by 
\[
X=((0,0), (0,100), (22,37), (36,41), (43,22), (100,0)) \subset \mathbb{R}^2
\]
and uniform weights. The package \texttt{LogConcDEAD} returns the heights 
\[
(y_1,y_2,y_3,y_4,y_5,y_6)= (-8.789569,-8.772087,-8.253580,-8.217959,-8.236983,-8.756922) 
\]
as the optimal solution. This gives rise to a triangulation of the convex hull of the data points with regions of linearity consisting of the triangles
\begin{equation*}
    \{1,2,3\}, \{1,3,5\}, \{1,5,6\}, \{2,3,4\}, \{2,4,6\}, \{3,4,5\},\{4,5,6\},
\end{equation*}
in Figure~\ref{fig43c}. This data gives an $\alpha$-value of $10^{26}$, which is much larger than the required $0.157671$. However, the system of equations it came from has a relatively high degree and the polynomial equations, when expanded, have between $929$ and $1564$ terms. We try to decrease the $\alpha$-value using the uniform sampling algorithm described above. We create a list of $ 729=3^6$ points in $\mathbb{R}^6$, consisting of all points whose i-th coordinate is
\[
y_i + \epsilon_i 2^{-14}, \epsilon_i \in \{-1,0,1\}.
\]
After a few repetitions, this finds a point with a lower alpha value. We repeat this process, each time decreasing the exponent of $2$ when creating the new test points. After 95 rounds we detect the refined point 
   $$
   \centering
   \begin{bmatrix}
        -8.789570552675578322471018111262921 \\
        -8.772086862481395608253513836856700 \\
        -8.253580886913590521217040193671505 \\
        -8.217957742357924329528595494315867 \\
        -8.236983233544571734253428918807660 \\
        -8.756919956247208359690046164738877
    \end{bmatrix}   $$
with alpha value $0.125519$. Therefore, this new solution is $\alpha$-certified. Note that this number has 34 decimal digits; we have rounded digits coming from the conversion from base $2$ ($109$ digits) after this position. Our conclusion is that the triangulation obtained by the heights in the \texttt{LogConcDEAD} output is certifiably correct.
\end{example}

\begin{example}
\label{finalexample}
We finish our paper by returning to our motivating example~\ref{motivatingexample} from the introduction, and consider two possible subdivisions of $P=\conv(X)$ for the regions of linearity of the optimal tent function: 
$$
        \Delta_1=\{\{1, 2, 3\},\{1, 3, 4\},\{2, 3, 4\},\{2,4,12\},\{1,4,8,11\},\{4,11,12\},\{8,11,12,13,14\}\}
$$
    and
$$
        \Delta_2=\{\{1, 2, 3\},\{1, 3, 4\},\{2,3,4,12\},\{1,4,8,12,13,14\}\}.
$$
The first subdivision $\Delta_1$ in Figure~\ref{fig:7-cell} arises from the \texttt{LogConcDEAD} output with default parameters after using the ``unique'' function. 
The second subdivision $\Delta_2$ in Figure~\ref{fig:4-cell} is given by the four regions of linearity  in Figure \ref{figure: 14 point} that we get by adjusting the precision in \texttt{LogConcDEAD} and then using the ``unique'' function.
Unfortunately the objective functions involved have too many summands for $\alpha$-certification to be feasible. 

\begin{figure*}
  \begin{subfigure}[b]{0.3\textwidth}
    \includegraphics[width=1\textwidth]{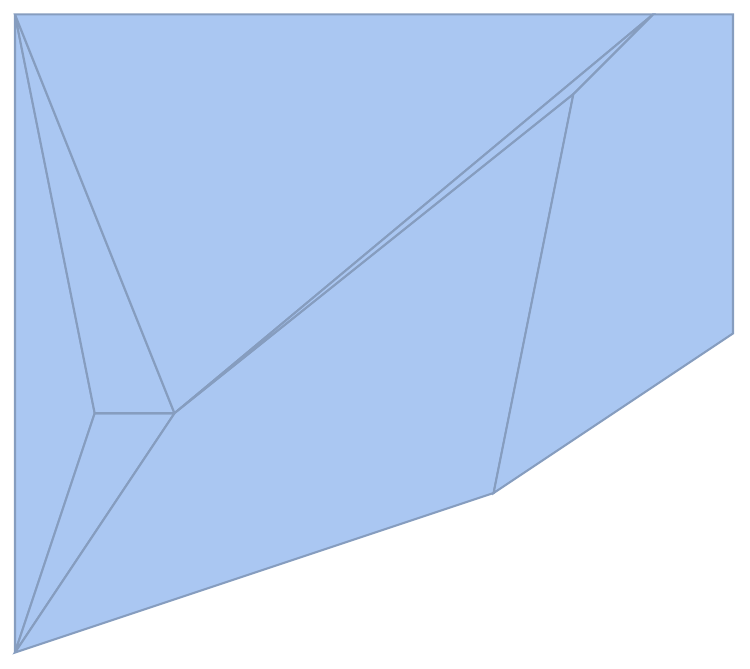}
    \caption{}
    \label{fig:7-cell}
  \end{subfigure}
  \hfill
  \begin{subfigure}[b]{0.3\textwidth}
    \includegraphics[width=1\textwidth]{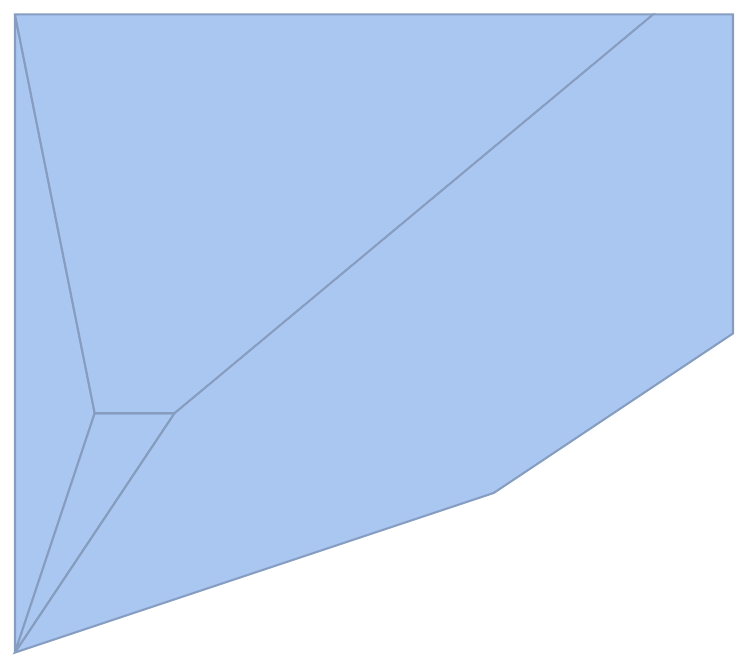}
    \caption{}
    \label{fig:4-cell}
  \end{subfigure}
  \hfill
  \begin{subfigure}[b]{0.3\textwidth}
    \includegraphics[width=1\textwidth]{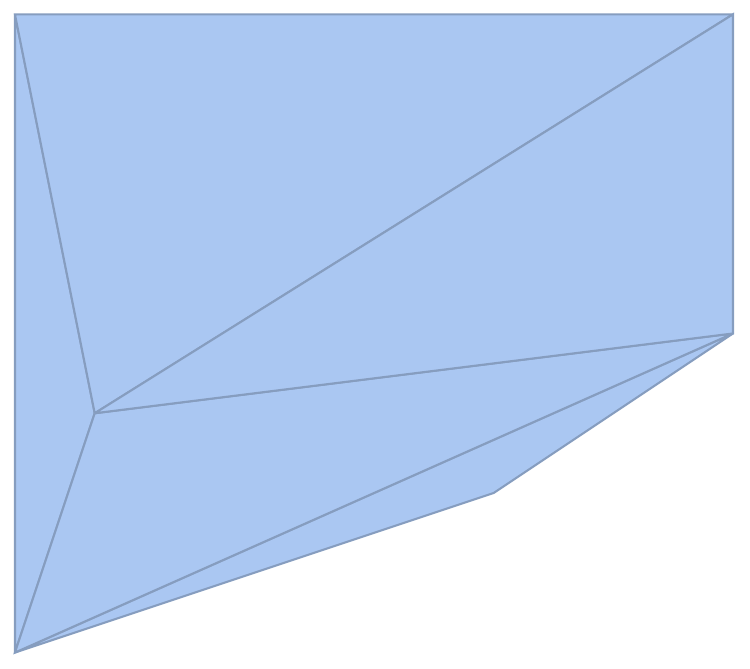}
    \caption{}
    \label{fig:7-cell-optimal}
  \end{subfigure}
  \caption{\label{heightgraphsfinal}Subdivisions in Example~\ref{finalexample}. (a) Subdivision $\Delta_1$ (b) Subdivision $\Delta_2$ (c) The subdivision induced by $y_{\Delta_1}^*$.}
\end{figure*}
    
As an alternative, we use the \texttt{NMaximize} command in \texttt{Mathematica} directly on the objective functions $S_{\Delta_1}$ and $S_{\Delta_2}$. The optimal $y_{\Delta_1}^*$ for the 7-cell subdivision gives a tent function whose regions of linearity are
$$
\{\{1,2,3\},\{1,8,13\},\{1,3,13\},\{2,3,14\},\{3,13,14\}\},
$$
which are depicted in Figure~\ref{fig:7-cell-optimal}. This triangulation is not refined by the subdivision $\Delta_1$: For example, the triangle $\{1,3,4\}$ in the subdivision $\Delta_1$ intersects the interiors of triangles $\{1,3,13\},\{2,3,14\},\{3,13,14\}$. Thus the 7-cell subdivision $\Delta_1$ is not the subdivision that we are looking for. In fact, the  vector $y_{\Delta_1}^*$ is not relevant, i.e. there exists $x_i$ such that $h_{X,y_{\Delta_1}^*}(x_i)>y_i$, and as a result $\int_P\exp(h_{X,y_{\Delta_1}^*}(t)) \neq 1$. 

The command \texttt{NMaximize} gives for the 4-cell subdivision 
\begin{align*}
y_{\Delta_2}^*=(&-4.32285, -4.7141, -4.2737, -4.14495, -4.26961, -4.10156, -3.94188, \\
&-3.91671, -3.94162, -3.80042, -3.76397, -3.68413, -3.69541, -3.62252).
\end{align*}
In comparison, the optimal height vector that we obtain using \texttt{LogConcDEAD} is
\begin{align*}
y^*=(&-4.322797, -4.714126, -4.273678, -4.144934, -4.269616, -4.101524, -3.941869,\\
&-3.916668, -3.941666, -3.800423, -3.764006, -3.684179, -3.695395, -3.622560).
 \end{align*}
A computation in \texttt{Polymake} verifies that $y_{\Delta_2}^*$ gives a tent function whose regions of linearity are exactly the cells of $\Delta_2$. This suggests that the 4-cell subdivision $\Delta_2$ is indeed the subdivision induced by the optimal $y^*$ in Example~\ref{motivatingexample}. We conclude with a haiku.
    \begin{center}
        Approximate heights,\\
        subdivisions inexact.\\
        A long road ahead.
    \end{center}

\end{example}
\vspace{0.5cm}

\noindent \textbf{Acknowledgements.} This project started at the Summer School on Geometric and Algebraic Combinatorics at Sorbonne University in June 2019. 
We thank Bernd Sturmfels for the guidance with the project, Gleb Pogudin for suggesting to use Lambert functions, Ricky Liu and Cynthia Vinzant for useful discussions and comments. 
Grosdos was partially supported by the DFG grant GK 1916, Kombinatorische Strukturen in der Geometrie. 
Kubjas and Kuznetsova were partially supported by the Academy of Finland Grant 323416. 
Scholten was partially supported by NSF Grant DMS 1620014.

\bibliographystyle{plain}
\bibliography{references}

\smallskip

\smallskip

\noindent Authors' affiliations:

\vspace{0.2cm}
\noindent Alexandros Grosdos, Institute for Mathematics, Osnabr\"{u}ck University,\\  \texttt{alexandros.grosdos@uni-a.de}

\vspace{0.2cm}
\noindent Alexander Heaton, Max Planck Institute for Mathematics in the Sciences, Leipzig, and Technische Universit\"at Berlin,\\ \texttt{alexheaton2@gmail.com}

\vspace{0.2cm}
\noindent Kaie Kubjas, Department of Mathematics and Systems Analysis, Aalto University,\\ \texttt{kaie.kubjas@aalto.fi}

\vspace{0.2cm}
\noindent Olga Kuznetsova,  Department of Mathematics and Systems Analysis, Aalto University, \texttt{olga.kuznetsova@aalto.fi}

\vspace{0.2cm}
\noindent Georgy Scholten, Department of Mathematics, North Carolina State University,\\ \texttt{georgy.scholten@lip6.fr}

\vspace{0.2cm}
\noindent Miruna-Stefana Sorea, Max Planck Institute for Mathematics in the Sciences, Leipzig,\\ 
\texttt{mirunastefana.sorea@ulbsibiu.ro}
\end{document}